\pdfoutput=1
\documentclass[11 pt]{amsart}
\usepackage{latexsym,amscd,amssymb, graphicx, amsthm, mathtools, young}  
\usepackage{blkarray}
\usepackage{tikz-cd, xcolor}

\usepackage[margin=1in]{geometry}
\usepackage{hyperref}
\numberwithin{equation}{section}

\newtheorem{theorem}{Theorem}[section]

\newtheorem{corollary}[theorem]{Corollary}
\newtheorem{lemma}[theorem]{Lemma}
\newtheorem{conjecture}[theorem]{Conjecture}

\newtheorem{example}[theorem]{Example}
\newtheorem{remark}[theorem]{Remark}

\theoremstyle{definition}
\newtheorem{defn}[theorem]{Definition}

\usepackage[backend=bibtex]{biblatex}
\usepackage{lipsum}
\usepackage{tikz}
\usepackage{ytableau}

\newcommand{\x}{\mathbf{x}}
\newcommand{\gr}{\mathrm{gr}}
\newcommand{\ZZ}{\mathbb{Z}}
\newcommand{\II}{\mathbf{I}}
\newcommand{\symm}{\mathfrak{S}}
\newcommand{\SYT}{\mathrm{SYT}}
\newcommand{\SSS}{\mathcal{S}}
\newcommand{\RR}{\mathcal{R}}
\newcommand{\topl}{\mathrm{Top}}
\newcommand{\nega}{\mathrm{neg}}
\newcommand{\pos}{\mathrm{pos}}
\newcommand{\lis}{\mathrm{lis}}
\newcommand{\Hilb}{{\mathrm {Hilb}}}
\newcommand{\CC}{\mathbb{C}}
\newcommand{\xx}{\mathbf{x}_{n \times n}}
\newcommand{\bl}{\boldsymbol{\lambda}}
\newcommand{\bm}{\boldsymbol{\mu}}
\newcommand{\bn}{\boldsymbol{\nu}}
\newcommand{\gl}{\mathrm{GL}}
\newcommand{\Res}{\mathrm{Res}}
\newcommand{\End}{\mathrm{End}}
\newcommand{\tl}{\tilde{\lambda}}
\newcommand{\tm}{\tilde{\mu}}
\newcommand{\sss}{\mathfrak{s}}

\begin{document}

\title{Viennot Shadows and Graded Module Structure in \\Colored Permutation Groups}

\author{Jasper M. Liu}
\address
{Department of Mathematics \newline \indent
University of California, San Diego \newline \indent
La Jolla, CA, 92093-0112, USA}
\email{mol008@ucsd.edu}

\begin{abstract}
    Let $\xx$ be a matrix of $n \times n$ variables, and let $\CC[\xx]$ be the polynomial ring on these variables. Let $\symm_{n,r}$ be the group of colored permutations, consisting of $n \times n$ complex matrices with exactly one nonzero entry in each row and column, where each nonzero entry is an $r$-th root of unity. We associate an ideal $I_{\symm_{n,r}} \subseteq \CC[\xx]$ with the group $\symm_{n,r}$, and use orbit harmonics to give an ideal-theoretic extension of the Viennot shadow line construction to $\symm_{n,r}$. This extension gives a standard monomial basis of $\CC[\xx]/I_{\symm_{n,r}}$, and introduces an analogous definition of ``longest increasing subsequence'' to the group $\symm_{n,r}$. We examine the extension of Chen's conjecture to this analogy. We also study the structure of $\CC[\xx]/I_{\symm_{n,r}}$ as a graded $\symm_{n,r} \times \symm_{n,r}$ module, which subsequently induces a graded $\symm_{n,r} \times \symm_{n,r}$ module structure on the $\CC$-algebra $\CC[\symm_{n,r}]$.
\end{abstract}


\maketitle

\section{Introduction}
Longest increasing subsequences in the symmetric group $\symm_n$ have been a well-studied object in combinatorics. In this paper, we apply orbit harmonics and ideal-theoretic extensions of Viennot's shadow line construction to introduce an analogous definition of longest increasing subsequence for the colored permutation group $\symm_{n,r}$.

Let $\symm_n$ be the symmetric group on $[n]=\{1,2,\dots,n\}$. Given $w \in \symm_n$, we write $w$ in its one line notation $w= [w(1),w(2),\dots,w(n)]$. An {\em increasing subsequence of length $l$} of $w$ is a sequence of integers $1 \leq i_1 < i_2 <  \dots  < i_l \leq n$ such that $w(i_1) < w(i_2) < \dots < w(i_l)$. We write 
\begin{equation}
    \lis(w) = \max \{k : w \text{ has an increasing subsequence of length } k\}
\end{equation}
for the length of the longest increasing subsequence of $w$. Let $a_{n,k}$ denote the number of permutations $w \in \symm_n$ with $\lis(w)=k$. For any positive integer $n$, Chen~\cite{chen} conjectured that the sequence $(a_{n,1},a_{n,2},\dots,a_{n,n})$ is log-concave, that is,
\begin{equation}
    a_{n,i} \cdot a_{n,i+2} \leq a_{n,i+1}^2 \text{ for } 1 \leq i \leq n-2.
\end{equation}

Let $\x$ be a finite set of variables, and let $\CC$ be the field of complex numbers. Let $\CC[\x]$ be the polynomial ring with variable set $\x$. If $I$ is a graded ideal in $\CC[\x]$, the {\em Hilbert series} of $\CC[\x]/I$ is 
\begin{equation}
    \Hilb( \CC[\x]/I; q) := \sum_{d \geq 0}  \dim_{\CC} (\CC[\x]/I)_d \cdot q^d,
\end{equation}
where $(\CC[\x]/I)_d$ is the degree $d$ part of the graded vector space $\CC[\x]/I$.

In ~\cite{rhoades2023increasing}, Rhoades studied the following ideal:
\begin{defn}\label{defn:defn1}
    Let $\xx = [x_{i,j}]_{1 \leq i,j \leq n}$ be an $n \times n$ matrix of variables, and let $\CC[\xx]$ be the polynomial ring in these variables. The ideal $I_n$ is the ideal generated by 
    \begin{itemize}
        \item the square $x_{i,j}^2$ for $1 \leq i,j \leq n$ of each variable;
        \item the product $x_{i,j} \cdot x_{i,j'}$ for $1 \leq i \leq n$ and $1 \leq j < j' \leq n$ of two distinct variables in the same row;
        \item the product $x_{i,j} \cdot x_{i',j}$ for $1 \leq i < i' \leq n$ and $1 \leq j \leq n$ of two distinct variables in the same column;
        \item the sum $\sum_{j=1}^{n}x_{i,j}$ for $1 \leq i \leq n$ of variables in the same row;
        \item the sum $\sum_{i=1}^{n}x_{i,j}$ for $1 \leq j \leq n$ of variables in the same column.
    \end{itemize}
\end{defn}

As all the generators of $I_n$ are homogeneous, $\CC[\xx]/I_n$ is a graded $\CC$-algebra. Let $\symm_n \times \symm_n$ act on $\xx$ by permuting rows and columns. This induces an action on $\CC[\xx]$ which stabilizes $I_n$, as the collection of generators of $I_n$ is stable under this action. Thus, $\CC[\xx]$ has the structure of a graded $\symm_n \times \symm_n$ module.

Rhoades proved~\cite{rhoades2023increasing} that $\CC[\xx]/I_n$ has Hilbert series
\begin{equation}
    \Hilb( \CC[\xx] / I_n; q) = a_{n,n} + a_{n,n-1} \cdot q + a_{n,n-2} \cdot q^2  + \dots + a_{n,1} \cdot q^{n-1},
\end{equation}
where coefficients of $q^k$ are the number of permutations $w$ in $\symm_n$ with $\lis(w)=n-k$. Using Viennot's shadow line construction, Rhoades attached a square-free monomial $\sss(w)$ for each $w \in \symm_n$, with the property that
\begin{equation}
    \deg(\sss(w)) = n - \lis(w),
\end{equation}
and showed that (Theorem~\ref{thm:thm1})
\begin{equation}
    \{\sss(w): w \in \symm_n  \}
\end{equation}
is a spanning set of the vector space $\CC[\xx]/I_n$, and proved that this is actually the standard monomial basis with respect to the Toeplitz order (see Definition~\ref{defn:defn3}). Rhoades also studied the graded and ungraded module structure of $\CC[\xx]/I_n$, giving the isomorphism
\begin{equation}
( \CC[\xx] / I_n )_{k} \cong \bigoplus_{\substack{\lambda \, \vdash \, n \\ \lambda_1  \, = \, n-k}} \End_{\CC}(V^{\lambda})
\end{equation}
and gave insights on the Novak-Rhoades conjecture~\cite{novak2020increasing}, which is a strengthening of Chen's conjecture.

We generalize Rhoades's work to colored permutation groups $\symm_{n,r}$, which is the group of $n \times n$ matrices with exactly one nonzero entry in each row or column, and all nonzero entries are $r$-th roots of unity. We specifically emphasize on the case $r=2$, where $\symm_{n,2}$ is denoted as $B_n$, the group of signed permutations. In Section~\ref{sec:sec1}, we will give more detailed descriptions of the properties of these groups. The following ideal is our object of study:

\begin{defn}\label{defn:defn2}
    $I_{\symm_{n,r}}$ is the ideal in $\CC[\xx]$ with generators:
    \begin{itemize}
        \item the $(r+1)$-th power $x_{i,j}^{r+1}$ for $1 \leq i,j \leq n$ of each variable;
        \item the product $x_{i,j} \cdot x_{i,j'}$ for $1 \leq i \leq n$ and $1 \leq j < j' \leq n$ of two distinct variables in the same row;
        \item the product $x_{i,j} \cdot x_{i',j}$ for $1 \leq i < i' \leq n$ and $1 \leq j \leq n$ of two distinct variables in the same column;
        \item the sum $\sum_{j=1}^{n}x_{i,j}^r$ for $1 \leq i \leq n$ of $r$-th powers of variables in the same row;
        \item the sum $\sum_{i=1}^{n}x_{i,j}^r$ for $1 \leq j \leq n$ of $r$-th powers of variables in the same column.
    \end{itemize}
    In particular, when $r=2$, we write $I_{\symm_{n,2}} =I_{B_n}$.
\end{defn}

Again, $I_{\symm_{n,r}}$ is a homogeneous ideal, so $\CC[\xx]/I_{\symm_{n,r}}$ is a graded $\CC$-algebra. Let $\gl_n(\CC) \times \gl_n(\CC)$ act on $\CC[\xx]$, where the action is induced by the following action on $\xx$: let $A,A' \in \gl_n(\CC)$,
\begin{equation}
    (A,A') \cdot \xx = A \xx A'^{-1}.
\end{equation}
Then, treating $\symm_{n,r}$ as a subgroup of $\gl_n(\CC)$, we obtain an action of $\symm_{n,r} \times \symm_{n,r}$ on $\CC[\xx]$. The ideal $I_{\symm_{n,r}}$ is stable under this action, so $\CC[\xx]/I_{\symm_{n,r}}$ has the structure of a graded $\symm_{n,r} \times \symm_{n,r}$ module. Using representation theoretical tools including the Branching Rule (Equation~\eqref{eq:branching}) and the Generalized Murnaghan-Nakayama Rule (Theorem~\ref{M-N}), we give the graded module structure (Theorem~\ref{thm:colored})
\begin{equation}
    (\CC[\xx]/I_{\symm_{n,r}})_{k} \cong \bigoplus_{\substack{  \bl  \vdash_r  n \\ r \lambda^0_1 + \sum_{i=1}^{r-1} i |\lambda^i| = rn-k  }} V^{\bl} \otimes V^{\bl'},
\end{equation}
where $V^{\bl}$ is the irreducible module of $\symm_{n,r}$ labelled by the $r$-tuple of partitions $\bl = (\lambda^0, \lambda^1, \dots , \lambda^{r-1})$ such that $\sum_{i=0}^r |\lambda^i|=n$, and $V^{\bl'}$ is the dual module of $V^{\bl}$. See Section~\ref{background} for details of relevant notations.

Focusing on the case of $B_n$, $b_{n,k}$ (first defined in Definition~\ref{defn:defn4}), which is analogous to $a_{n,k}$, counts the number of signed permutations $w$ in $B_n$ with a property analogous to ``having a longest increasing subsequence of length $k$". For $n \geq 2$, the Hilbert series of $\CC[\xx]/I_{B_n}$ is
\begin{equation}
    \Hilb( \CC[\xx] / I_{B_n}; q) = b_{n,2n} + b_{n,2n-1} \cdot q + b_{n,2n-2} \cdot q^2  + \dots + b_{n,2} \cdot q^{2n-2}.
\end{equation}

By testing $n$ up tp $40$, we observe that the sequence $\{b_{n,k}\}_{2 \leq k \leq 2n}$ is ``almost" log-concave: log-concavity holds for $n \leq 8$; and for larger $n$'s, log-concavity breaks for some small $k$'s and then holds for the remaining majority of the sequence. 

We also extend this notion of ``having longest increasing subsequence of length $k$" to elements of $\symm_{n,r}$, and we define (Definition~\ref{defn:c_n,r,k}) $c_{n,r,k}$ that counts the number of elements in $\symm_{n,r}$ with this property. We prove (Theorem~\ref{thm:hilb_n,r}) that the Hilbert series of $\CC[\xx]/I_{\symm_{n,r}}$ is 
\begin{equation}
    \Hilb(\CC[\xx]/I_{\symm_{n,r}};q) = \sum_{k=1}^{rn} c_{n,r,k}q^{rn-k},
\end{equation}
and also give the conjecture (Conjecture~\ref{conj}) that the sequence $\{c_{n,r,k}\}$ is uni-modal, that is, there exists $k$ such that
\begin{equation}
    c_{n,r,d} \leq c_{n,r,d+1} \text{ for } d < k \text{ and } c_{n,r,d} \geq c_{n,r,d+1} \text{ for } d \geq k.
\end{equation}

The rest of the paper is structured as follows. In Section~\ref{background}, we introduce the background material on Gröbner theory, orbit harmonics, group structure of $\symm_{n,r}$, the Schensted correspondence, and representation theory. In Section~\ref{SMB}, we generalize Viennot's shadow line construction to the group $\symm_{n,r}$, and use this construction to find a standard monomial basis of $\CC[\xx]/I_{\symm_{n,r}}$. In Section~\ref{Module} we give the graded $\symm_{n,r} \times \symm_{n,r}$ module structure of $\CC[\xx]/I_{\symm_{n,r}}$. And in Section~\ref{Conjecture}, we examine the generalization of Chen's conjecture, and close with possible directions of future research.

\section{Background}~\label{background}
\subsection{Gröbner Theory}
Let $\x_N = (x_1 , x_2 , \dots , x_N)$ be a finite list of variables, and let $\CC[\x_N]$ be the polynomial ring over these variables. A total order $<$ on the set of monomials in $\CC[\x_N]$ is called a \textit{monomial order} if 

\begin{itemize}
    \item $1 \leq m$ for any monomial $m$,
    \item for any monomials $m_1, m_2, m_3$, we have that $m_1 < m_2$ implies $m_1 m_3 < m_2 m_3$.
\end{itemize}

Given any polynomial $f$ in $\CC[\x_N]$, the {\em initial monomial} $\textrm{in}_{<}f$ of $f$ is defined to be the largest monomial term in $f$ with respect to $<$. The \textit{initial ideal} of an ideal $I$ in $\CC[\x_N]$ is the ideal generated by the initial monomials of all polynomials in $I$, i.e.
\begin{equation}
\textrm{in}_< I = \langle \textrm{in}_< f : f \in I \rangle 
\end{equation}

A monomial in $\CC[\xx]$ which is not an element of $\textrm{in}_{<} I$ is called a {\em standard monomial}. It is known that 
\begin{equation}
\{m+ I : m \text{ is a standard monomial} \} 
\end{equation}
forms a basis of the vector space $\CC[\x_N]/I$, and this is called the \textit{standard monomial basis}.
More detailed descriptions of Gröbner theory can be found in~\cite{10.5555/1951670}.

\subsection{Orbit Harmonics}
Given $Z$ a finite locus of points in $\CC^N$, the {\em vanishing ideal} $\II(Z) \subseteq \CC[\x_N]$ is the collection of all polynomials $f$ satisfying $f(x) = 0$ for all $x$ in $Z$. It is known that we have the isomorphism of vector spaces
\begin{equation}\label{eq:eqn1}
\CC[\x_N]/\II(Z) \cong \CC[Z],
\end{equation}
where we see $\CC[Z]$ as the vector space of functions from $Z$ to $\CC$. Note that the vector space $\CC[\x_N]/\II(Z)$ is usually ungraded. 

Given a polynomial $f$ in $\CC[\x_N]$, we can write $f = f_0 + f_1 + \dots +f_d$ where $f_i$ is homogeneous of degree $i$ and $f_d$ is nonzero. Denote $\tau(f)$ as the highest degree homogeneous part of $f$, i.e. $\tau(f) = f_d$. Given an ideal $I$, the \textit{associated graded ideal} $\gr I$ is defined to be 
\begin{equation}
\gr I = \langle \tau(f):f \in I \rangle.
\end{equation}

The isomorphism~\eqref{eq:eqn1} of ungraded vector spaces above can be further extended to 
\begin{equation}\label{eq:eqn2}
\CC[\x_N]/\II(Z) \cong \CC[Z] \cong \CC[\x_N]/\gr \II(Z) ,
\end{equation}
and since $\tau(f)$ is homogeneous for all $f$, $\CC[\x_N]/\gr I$ has an additional structure as a graded vector space.

Orbit harmonics has seen applications to presenting cohomology rings \cite{GARSIA199282}, Macdonald theory \cite{HRS}, cyclic sieving \cite{OR}, Donaldson-Thomas theory \cite{RRT}, and Ehrhart theory \cite{reiner2024harmonicsgradedehrharttheory}. One expects algebraic properties of $\CC[\x_N]/\gr \II(Z)$ to be governed by combinatorial properties of $Z$.

\subsection{The Schensted Correspondence}
Let $n$ be any positive integer. A \textit{partition} $\lambda$ of $n$ is a weakly decreasing sequence of positive integers $\lambda = (\lambda_1, \lambda_2, \dots , \lambda_r)$ such that $\sum_{i=1}^{r} \lambda_i = n$. We use the notation $\lambda \vdash n$ to indicate that $\lambda$ is a partition of $n$. 

A \textit{Young diagram} of shape $\lambda$ is a diagram of boxes with $\lambda_i$ boxes in the $i$-th row. A \textit{standard Young tableau} of shape $\lambda$ is a filling of $\{1,2,\dots ,n\}$ into the Young diagram of $\lambda$ such that each number appears exactly once, and that the entries are increasing across rows and down columns. For example, given $\lambda = (4,2,1) \vdash 7$, we give the Young diagram of $\lambda$ and one standard Young tableau of shape $\lambda$.
\begin{center}
    
\begin{ytableau}
    \: & \: & \: & \: \cr
       & \cr
    \:
\end{ytableau} \quad \quad
\begin{ytableau}
    1 & 3 & 6 & 7 \cr
    2 & 5 \cr
    4
\end{ytableau}

\end{center}

We write $\SYT(\lambda)$ to denote the collection of standard Young tableaux of shape $\lambda$. The {\em Schensted correspondence}~\cite{schensted_1961} is a bijection:
\begin{equation}
    \symm_n \xrightarrow{\quad \sim \quad} \bigsqcup_{\lambda \vdash n} \{ (P,Q): P,Q \in \SYT (\lambda)\}, 
\end{equation}
that is, the Schensted correspondence associates to each $w \in \symm_n$ an ordered pair $(P,Q)$ of standard Young tableaux of the same shape. The explicit bijection is usually given by an insertion algorithm, details of which can be found in ~\cite{sagan2013symmetric}.

\subsection{Signed and Colored Permutation Groups}\label{sec:sec1}
Recall that the group $\symm_{n,r}$ consists of all $n \times n$ complex matrices with exactly one nonzero entry in each row and column, and each nonzero entry is an $r$-th root of unity. We introduce a more combinatorial interpretation of the group in this section:

The \textit{$r$-colored permutation group} $\symm_{n,r}$ is defined to be the wreath product $(\ZZ/r\ZZ) \wr \symm_n$. In particular, when $r=2$, the group is called the \textit{signed permutation group}, and we denote the group as $B_n$. An element $w$ of $\symm_{n,r}$ can be viewed as a pair $(\sigma, \kappa)$, where $\sigma$ is a permutation in $\symm_n$, and $\kappa : [n] \rightarrow \{0,1,\dots,r-1\}$ is a coloring of elements in $[n]$. Writing $[n]^r$ for the set $\{i^j : i \in [n], j \in \{0,1,\dots,r-1\} \}$, we can see $w=(\sigma,\kappa)$ as a function $[n]^r \rightarrow [n]^r$ via the rule
\begin{equation}
    (w)(i^j) = \sigma(i)^{\kappa(\sigma(i))+j},
\end{equation}
where addition is modulo $r$. An example is when $n=5, r=3$, and $w= (\sigma,\kappa)$ where $\sigma=[4,2,5,3,1]=(1435)(2), \kappa=(2,1,2,2,0)$, then we have $w(1^0) = 4^2,w(2^0) = 2^1,w(3^0) = 5^0,w(4^0) = 3^2,w(5^0) = 1^2$. In one-line notation, we denote $w=[4^2,2^1,5^0,3^2,1^2]$, and we write the cycle notation of $w$ as $(1^2,4^2,3^2,5^0)(2^1)$. For a cycle in the cycle decomposition of $w$, we refer to the \textit{color} of the cycle to be the sum of colors of elements in it, reduced modulo $r$. For example, the color of $(1^2,4^2,3^2,5^0)$ is $2+2+2 = 0 \mod 3$, and the color of $(2^1)$ is $1$.

\begin{defn}
    Given integers $n \geq 0$, $r \geq 1$, an {\em $r$-partition} of $n$ is a sequence of $r$ partitions $\boldsymbol{\lambda} = (\lambda^0, \lambda^1, \dots , \lambda^{r-1})$ such that $\sum_{i=0}^{r-1} |\lambda^i| = n$. We write $\bl \vdash_r n$ for $\bl$ being an $r$-partition of $n$. In particular, when $r=2$, we call such $r$-partitions {\em bipartitions}.
\end{defn}

\begin{defn}
    The {\em cycle type} of an element $w=(\sigma,\kappa) \in \symm_{n,r}$ is the $r$-partition $\boldsymbol{\lambda}=(\lambda^0,\lambda^1,\dots,\lambda^{r-1})$ such that $\lambda^i$ consists of lengths of cycles in $w$ of color $i$.
\end{defn}
Two elements in $\symm_{n,r}$ are conjugates if and only if they have the same cycle type. Thus, the conjugacy classes of $\symm_{n,r}$ are labelled by $r$-partitions of $n$.

\subsection{Representation Theory}

\subsubsection{$\symm_n$ Representations}
Let $\symm_n$ be the symmetric group. As the conjugacy classes of $\symm_n$ are labelled by partitions of $n$, there is a bijection between partitions of $n$ and irreducible representations of $\symm_n$ over $\CC$. For $\lambda \vdash n$, we denote $V^\lambda$ as the irreducible representation of $\symm_n$ associated to $\lambda$. The dimension of $V^\lambda$ is given by $|\SYT(\lambda)|$, the number of standard Young tableau of shape $\lambda$.

\subsubsection{Representations of Signed and Colored Permutations}
Let $\boldsymbol{\lambda}=(\lambda^0, \lambda^2, \dots , \lambda^{r-1}) \vdash_r n$. A standard Young tableau of shape $\boldsymbol{\lambda}$ is a bijection from $[n]$ to the collection of boxes in the disjoint union of Young diagrams of $\lambda^0$, $\lambda^1$, \dots,$\lambda^{r-1}$ such that the entries in each diagram are increasing across rows and down columns. Below is an example of a standard Young tableau for bipartition $((2,1),(4,2)) \vdash 9$. 
\begin{center}
    \begin{ytableau}
        1 & 5 \cr
        7 \cr 
    \end{ytableau} \quad \quad
    \begin{ytableau}
        2 & 4 & 8 & 9 \cr
        3 & 6 \cr
    \end{ytableau}
\end{center}
Specht constructed~\cite{Specht1932} a bijection between the irreducible representations of $\symm_{n,r}$ and $r$-partitions of $n$ such that the dimension of the irreducible representation labelled by $\bl$ is the number of standard Young tableaux of shape $\bl$.

\subsection{Branching Rule and Murnaghan-Nakayama Rule}

Let $\lambda \vdash n_1$ and $\mu \vdash n_2$, with the property that $\lambda_i \geq \mu_i$ for all $i$, and let $n=n_1 - n_2$. The {\em skew partition $\lambda / \mu$} has diagram obtained by deleting the diagram of $\mu$ from the diagram of $\lambda$. For example, when $\lambda = (4,3,2) \vdash 9$ and $\mu = (2,1) \vdash 3$, the skew partition $\lambda / \mu$ has diagram drawn below.
\begin{center}
    \begin{ytableau}
      \none  & \none  & \: & \: \cr
      \none  &\: & \:  \cr
     \: & \:
    \end{ytableau}
\end{center}

The $\symm_n$ character $\chi^{\lambda/\mu}$ is defined to be
\begin{equation}
    \chi^{\lambda/\mu} = \sum_{|\nu| = n}  c^{\lambda}_{\mu, \nu} \chi^{\nu}
\end{equation}
where the sum is over all partitions $\nu$ of $n=n_1-n_2$ and $c^{\lambda}_{\mu,\nu}$ is the Littlewood-Richardson coefficient.

More generally, if $\bl = (\lambda^0 , \lambda^1, \dots ,\lambda^{r-1})$ and $\bm = (\mu^0, \mu^1, \dots ,\mu^{r-1})$ are $r$-partitions such that $\lambda^i/ \mu^i$ are all well-defined skew partitions, the {\em $r$-skew partition $\bl/\bm$} is then
\begin{equation}
    \bl/\bm = (\lambda^0/\mu^0, \lambda^1/\mu^1, \dots , \lambda^{r-1}/\mu^{r-1}).
\end{equation} 
The size of $\bl/\bm$ is $| \bl/\bm | = \sum_{i=0}^{r-1}|\lambda^i|-|\mu^i|$.

Given an $r$-skew partition $\bl/\bm$ such that $|\bl/\bm|=n$, we also have a character $\chi^{\bl/\bm}$ of $\symm_{n,r}$ associated to it. It is defined to be 
\begin{equation}
    \chi^{\bl/\bm} = \sum_{\bn \vdash_r n}  c^{\bl}_{\bm, \bn} \chi^{\bn}
\end{equation}
where the sum is over all $r$-partitions $\bn=(\nu^1,\nu^2,\dots,\nu^r) \vdash_r n$, and $c^{\bl}_{\bm, \bn}$ is the {\em generalized Littlewood-Richardson coefficient}, defined as
\begin{equation}
    c^{\bl}_{\bm, \bn}=\prod_{i=1}^{r} c^{\lambda^i}_{\mu^i, \nu^i},
\end{equation}
which is the product of Littlewood-Richardson coefficients.

\subsubsection{Branching Rule}
The characters labelled by $r$-skew partitions play an important role in the branching rule of $\symm_{n,r}$ representations. Let $\bl$ be an $r$-partition of $\symm_{n,r}$. For $k < n$, $\symm_{k,r} \times \symm_{n-k,r}$ is a natural subgroup of $\symm_{n,r}$, and we have the branching rule
\begin{equation}\label{eq:branching}
    \Res^{\symm_{n,r}}_{\symm_{k,r} \times \symm_{n-k,r}} \chi^{\bl} = \sum_{\bm} \chi^{\bm} \times \chi^{\bl/\bm}
\end{equation}
where the sum is over all $r$-partitions $\bm \vdash_r k$ such that $\bl/\bm$ is well defined. The proof of the branching rule can be found in~\cite{pjm/1102646930}. 

\subsubsection{Generalized Murnaghan-Nakayama Rule}
We then recall how to evaluate $\chi^{\bl/\bm}$ on $\symm_{n,r}$. A {\em ribbon} is a skew partition with a rookwise connected diagram that does not contain a $2 \times 2$ square as a sub-diagram, and the {\em height} of a ribbon is defined as the number of rows in it minus $1$. The diagram at the beginning of the section is an example of a ribbon of height $2$. We say that an $r$-skew partition $\bl/\bm$ is a {\em $r$-ribbon} if only one of the skew partitions $\lambda^i/\mu^i$ is non-empty and the non-empty skew partition is a ribbon. If $\bl/\bm$ is an $r$-ribbon with $\lambda^j/\mu^j$ non-empty, we say $\theta$ is the {\em character of $\bl/\bm$}, where $\theta$ is the irreducible character of $\ZZ/r\ZZ$
\begin{equation}
    \theta : \ZZ/r\ZZ \rightarrow \gl _1(\CC)  \quad \quad 
    \theta(z) = e^{2j \pi i z/r}.
\end{equation}

A decomposition of $\bl/\bm$ into $r$-ribbons is a nested sequence of $r$-partitions
\begin{equation}\label{nested-seq}
    \bm = \bl_0 \subseteq \bl_1 \subseteq \dots \subseteq \bl_l = \bl 
\end{equation}
where $\bl_i/\bl_{i-1}$ is an $r$-ribbon for $1 \leq i \leq l$. Let $\beta_i = |\bl_i/\bl_{i-1}|$ and $\beta = (\beta_1,\beta_2,\dots,\beta_l)$, we call the decomposition in~\eqref{nested-seq} a {\em $\beta$-decomposition}. With this definition, we introduce the generalized Murnaghan-Nakayama rule.

\begin{theorem}\label{M-N}
    Let $w$ be an element of $\symm_{n,r}$. Suppose the cycles of $x$ have lengths $\beta_1,\beta_2,\dots,\beta_l$, and the cycles have colors $c_1,c_2,\dots,c_l$. Then, we have the evaluation of $\chi^{\bl/\bm}$ on $w$:
    \begin{equation}\label{eq:M-N}
        \chi^{\bl/\bm}(w)= \sum_{\bl'} \prod_{i=1}^{l}(-1)^{h_i} \theta_i (c_i)
    \end{equation}
where the sum is over $\bl'$ of $\beta$-decompositions of $\bl/\bm$, $h_i$ is the height of the $r$-ribbon in the $i$-th step of the $\beta$-decomposition, and $\theta_i$ is the character of the $r$-ribbon in the $i$-th step as defined above.
\end{theorem}

The proof of Theorem~\ref{M-N} can also be found in~\cite{pjm/1102646930}.

\section{Standard Monomial Basis}\label{SMB}
\subsection{Viennot Shadow}
Given a permutation $w \in \symm_n$, we associate to it a diagram on a $n \times n$ grid, labelling all points $(i,w(i))$ on the grid. Below is the diagram of $w = [5,1,3,6,7,2,4] \in \symm_7$. 
\begin{center}
    \begin{tikzpicture}[x=1.5em,y=1.5em]
    \draw[step=1,black,thin] (0,0) grid (6,6);
    \filldraw [blue] (0,4) circle (2pt);
    \filldraw [blue] (1,0) circle (2pt);
    \filldraw [blue] (2,2) circle (2pt);
    \filldraw [blue] (3,5) circle (2pt);
    \filldraw [blue] (4,6) circle (2pt);
    \filldraw [blue] (5,1) circle (2pt);
    \filldraw [blue] (6,3) circle (2pt);
    \end{tikzpicture} 
\end{center}

Viennot proved~\cite{viennot} that the diagram can be used to construct the pair of standard Young tableau in the Schensted correspondence. Imagine a light source at the bottom left corner which shines northeast. Each point $(i,w(i))$ in the diagram blocks light to its northeast. Consider the boundary of the shadow region, and call it the {\em first shadow line}. The first shadow line for the permutation $w$ above is depicted in the leftmost diagram below. Then, remove all points on the first shadow line and iterate the process, we get the {\em second shadow line}, {\em the third shadow line}, and so on. The diagram below in the middle depicts all shadow lines of permutation $w = [5,1,3,6,7,2,4]$. 

\begin{center}
    \begin{tikzpicture}[x=1.4em,y=1.4em, very thick,color = blue]
    \draw[step=1,black,thin] (0,0) grid (6,6);
    \filldraw [black] (0,4) circle (2pt);
    \filldraw [black] (1,0) circle (2pt);
    \filldraw [black] (2,2) circle (2pt);
    \filldraw [black] (3,5) circle (2pt);
    \filldraw [black] (4,6) circle (2pt);
    \filldraw [black] (5,1) circle (2pt);
    \filldraw [black] (6,3) circle (2pt);
    \draw(0,7)--(0,4)--(1,4)--(1,0)--(7,0);
    \end{tikzpicture} \quad \quad
    \begin{tikzpicture}[x=1.4em,y=1.4em, very thick,color = blue]
    \draw[step=1,black,thin] (0,0) grid (6,6);
    \filldraw [black] (0,4) circle (2pt);
    \filldraw [black] (1,0) circle (2pt);
    \filldraw [black] (2,2) circle (2pt);
    \filldraw [black] (3,5) circle (2pt);
    \filldraw [black] (4,6) circle (2pt);
    \filldraw [black] (5,1) circle (2pt);
    \filldraw [black] (6,3) circle (2pt);
    \draw(0,7)--(0,4)--(1,4)--(1,0)--(7,0);
    \draw(2,7)--(2,2)--(5,2)--(5,1)--(7,1);
    \draw(3,7)--(3,5)--(6,5)--(6,3)--(7,3);
    \draw(4,7)--(4,6)--(7,6);
    \end{tikzpicture} \quad \quad
    \begin{tikzpicture}[x=1.4em,y=1.4em, very thick,color = blue]
    \draw[step=1,black,thin] (0,0) grid (6,6);
    \filldraw [black] (0,4) circle (2pt);
    \filldraw [black] (1,0) circle (2pt);
    \filldraw [black] (2,2) circle (2pt);
    \filldraw [black] (3,5) circle (2pt);
    \filldraw [black] (4,6) circle (2pt);
    \filldraw [black] (5,1) circle (2pt);
    \filldraw [black] (6,3) circle (2pt);
    \draw(0,7)--(0,4)--(1,4)--(1,0)--(7,0);
    \draw(2,7)--(2,2)--(5,2)--(5,1)--(7,1);
    \draw(3,7)--(3,5)--(6,5)--(6,3)--(7,3);
    \draw(4,7)--(4,6)--(7,6);
    \filldraw [red] (1,4) circle (2pt);
    \filldraw [red] (5,2) circle (2pt);
    \filldraw [red] (6,5) circle (2pt);
    \end{tikzpicture}
    
\end{center}
Suppose the shadow lines of $w$ are $L_1,L_2,\dots,L_k$, and we have $w \rightarrow (P(w),Q(w))$ under the Schensted correspondence, Viennot proved~\cite{viennot} the $y$ coordinates of the infinite horizontal rays of $L_1, L_2,\dots ,L_k$ form the first row of $P(w)$, and the $x$ coordinates of the infinite vertical rays of $L_1, L_2,\dots ,L_k$ for the first row of $Q(w)$. In the example above, we have the first row of $P(w)$ is $1,2,4,7$; and the first row of $Q(w)$ is $1,3,4,5$.

\begin{defn}
    Given the permutation $w \in \symm_n$, the {\em shadow set $\SSS(w)$} of $w$ is the collection of points that lie on the northeast corners of the shadow lines.
\end{defn}
In the example above, we have that for $w = [5,1,3,6,7,2,4] \in \symm_7$, $\SSS(w) = \{(2,5),(6,3),(7,6)\}$.
\begin{remark}
    From the construction of shadow lines, we can see that the shadow set of $w$ will contain no two points in the same row or column, which will induce the following definition.
\end{remark}

\begin{defn}
    A subset $\RR$ of points in the $n \times n$ diagram is a {\em (non-attacking) rook placement} if it contains at most one point in each row or column.
\end{defn}

As mentioned in the remark, the shadow set of any permutation is a rook placement. However, not every rook placement is the shadow set of a permutation. An algorithm for identifying whether a rook placement is the shadow set of a permutation is specified in~\cite{rhoades2023increasing}. 

After obtaining the shadow set of a permutation, we can iterate the shadow line construction on $\SSS(w)$, as depicted in the diagram below on the left. Viennot proved~\cite{viennot} that the $y$ coordinates of the infinite horizontal rays of the new shadow lines correspond to the second row of $P(w)$, and the $x$ coordinates of the infinite vertical rays correspond to the second row of $Q(w)$. Keep iterating this process on the new shadow set, and we will get the third row of the pair of tableau. The process ends when the new shadow set is empty, and we would have the complete pair of tableaux $(P(w),Q(w))$.

\begin{center}
    \begin{tikzpicture}[x=1.4em,y=1.4em, very thick,color = blue]
    \draw[step=1,black,thin] (0,0) grid (6,6);
    \filldraw [black] (1,4) circle (2pt);
    \filldraw [black] (5,2) circle (2pt);
    \filldraw [black] (6,5) circle (2pt);
    \filldraw [red] (5,4) circle (2pt);
    \draw(1,7)--(1,4)--(5,4)--(5,2)--(7,2);
    \draw(6,7)--(6,5)--(7,5);
    \end{tikzpicture} \quad \quad
    \begin{tikzpicture}[x=1.4em,y=1.4em, very thick,color = blue]
    \draw[step=1,black,thin] (0,0) grid (6,6);
    \filldraw [black] (5,4) circle (2pt);
    \draw(5,7)--(5,4)--(7,4);
    \end{tikzpicture}
\end{center}

Given the connections between the shadow sets and the Schensted correspondence, the following lemma is immediate:
\vspace{0.1in}
\begin{lemma}
~\label{lem:size}
    Let $w \in \symm_n$. The size of the shadow set is $|\SSS(w)| = n - \lis(w)$, where $\lis(w)$ is the length of the longest increasing subsequence of $w$, as defined in the introduction.
\end{lemma}

\subsection{Shadow monomials}
Given a collection of points $S \in [n] \times [n]$, we denote $m(S)$ to be the square-free monomial $\prod_{(i,j) \in S}x_{i,j}$. Let $I_n$ be defined as in the introduction. It is straightforward that 
\begin{equation*}
    \{ m(\RR) : \RR \text{ is a rook placement} \}
\end{equation*}
is a spanning set of $\CC[\xx]/I_n$, since the product of any two variables in the same row or column is in $I_n$.

Recall that given a set of variables $y_1,y_2,\dots,y_N$, the {\em lexicographical order} is defined such that $y_1 > y_2 > \dots > y_N$ and $y_1^{a_1} y_2^{a_2} \dots y_N^{a_n} < y_1^{b_1} y_2^{b_2} \dots y_N^{b_N}$ if and only if there exists $j \leq n$ such that $a_i = b_i$ for all $i < j$, and $a_j < b_j$. The standard monomial basis of $\CC[\xx]/I_n$ was computed in~\cite{rhoades2023increasing} with respect to the monomial order defined as follows. 

\begin{defn}\label{defn:defn3}
    The {\em Toeplitz order}, denoted by $<_{\topl}$, is the lexicographical order with respect to the following ordering on the variables:
    \begin{equation}
        x_{1,1} > x_{2,1} > x_{1,2} > x_{3,1} > x_{2,2} > x_{1,3} > \dots > x_{n,n-1} > x_{n-1,n} > x_{n,n}.
    \end{equation}    
\end{defn}

Rhoades defined the shadow monomials as follows:
\begin{defn}
    Given a permutation $w \in \symm_n$, the {\em shadow monomial} of $w$, is the square-free monomial
    \begin{equation}
        \sss(w)=m(\SSS(w)).
    \end{equation}
\end{defn}

Rhoades then proved~\cite[Lemma 3.11]{rhoades2023increasing}

\begin{theorem}
    \label{thm:thm1}
    The monomials $\{\sss(w) : w \in \symm_n \}$ span the vector space $\CC[\xx]/I_n$.
\end{theorem}

Then, identifying $\symm_n$ as the group of $n \times n$ permutation matrices (i.e. matrices such that all entries are $0$ and $1$, with exactly one $1$ in each row and column), we have that the vanishing ideal of $\symm_n$ is generated by
\begin{itemize}
    \item $x_{i,j}^2-x_{i,j}$ for $1 \leq i,j \leq n$;
    \item $x_{i,j} \cdot x_{i,j'}$ for $1 \leq i \leq n$, $1 \leq j < j' \leq n$;
    \item $x_{i,j} \cdot x_{i',j}$ for $1 \leq j \leq n$, $1 \leq i < i' \leq n$;
    \item $x_{i,1}+x_{i,2}+\dots+x_{i,n}-1$ for $1 \leq i \leq n$;
    \item $x_{1,j}+x_{2,j}+\dots+x_{n,j}-1$ for $1 \leq j \leq n$.
\end{itemize}

A quick observation is that $I_n \subseteq \gr \II(\symm_n)$, since $I_n$ is generated by the highest degree homogeneous parts of the generators of $\II(\symm_n)$. The orbit harmonics~\eqref{eq:eqn2} gives
\begin{equation}
    \CC[\symm_n]  \cong \CC[\xx]/ \II(\symm_n) \cong \CC[\xx]/\gr \II(\symm_n).
\end{equation}
Using an argument on the dimensions of the vector spaces, Rhoades showed~\cite{rhoades2023increasing} that the ideals $I_n$ and $\gr \II(\symm_n)$ are actually the same, and showed that 
\begin{equation*}
    \{\sss(w) : w \in \symm_n \}
\end{equation*}  
is the standard monomial basis of $\CC[\xx]/I_n$ with respect to the Toeplitz order. Thus, let $a_{n,k}$ be the number of permutations $w$ in $\symm_{n}$ with $\lis(w) = k$, the Hilbert series of $\CC[\xx]/I_n$ is
\begin{equation}
    \Hilb( \CC[\xx] / I_n; q) = a_{n,n} + a_{n,n-1} \cdot q + a_{n,n-2} \cdot q^2  + \dots + a_{n,1} \cdot q^{n-1}.
\end{equation}

\subsection{Generalization to Signed Permutations}
\subsubsection{Spanning Set}
Given a signed permutation $w=(\sigma,\kappa)$, we define its diagram on the $n \times n$ grid. A point $(i,j)$ on the grid is colored green if $\sigma(i)=j$ and $\kappa(j)=1$. It is colored blue if $\sigma(i)=j$ and $\kappa(j)=0$. For example, when $n=6$ and the one-line notation of $w$ is $[2^1,5^0,3^0,1^0,6^0,4^1]$, the diagram is below on the left:
\begin{center}
    \begin{tikzpicture}[x=1.4em,y=1.4em, very thick,color = blue]
    \draw[step=1,black,thin] (0,0) grid (5,5);
    \filldraw [green] (0,1) circle (2pt);
    \filldraw [blue] (1,4) circle (2pt);
    \filldraw [blue] (2,2) circle (2pt);
    \filldraw [blue] (3,0) circle (2pt);
    \filldraw [blue] (4,5) circle (2pt);
    \filldraw [green] (5,3) circle (2pt);
    \end{tikzpicture} \quad \quad
    \begin{tikzpicture}[x=1.4em,y=1.4em, very thick,color = brown]
    \draw[step=1,black,thin] (0,0) grid (5,5);
    \filldraw [green] (0,1) circle (2pt);
    \filldraw [blue] (1,4) circle (2pt);
    \filldraw [blue] (2,2) circle (2pt);
    \filldraw [blue] (3,0) circle (2pt);
    \filldraw [blue] (4,5) circle (2pt);
    \filldraw [green] (5,3) circle (2pt);
    \draw(1,6)--(1,4)--(2,4)--(2,2)--(3,2)--(3,0)--(6,0);
    \draw(4,6)--(4,5)--(6,5);
    \filldraw [red] (2,4) circle (2pt);
    \filldraw [red] (3,2) circle (2pt);
    \end{tikzpicture}
    
\end{center}
We define the {\em negative set} of $w$ as 
\begin{equation}
    \nega(w)=\{(i,j):\sigma(i)=j \text{ and } \kappa(j)=1\ \}.
\end{equation} or equivalently, it's the collection of points colored red in the diagram. Similarly, we define the {\em positive set} of $w$  to be 
\begin{equation}
    \pos(w)=\{(i,j) : \sigma(i)=j \text{ and } \kappa(j)=0\}.
\end{equation}

\begin{remark}\label{rk:rk1}
    It is an immediate result that $\nega(w)$ is a rook placement, and for any rook placement $\RR$, there is $w \in B_n$ such that $\nega(w) = \RR$.
\end{remark}

Then, we do the shadow line construction on the positive set of $w$, and obtain the shadow set $\SSS(\pos(w))$, as depicted above on the right. The shadow lines are drawn in brown, and the points in the shadow set are labelled green. With this construction, we associate to the signed permutation $w$ a monomial 
\begin{equation}
    \sss(w) = m(\nega(w)) \cdot m(\SSS(\pos(w)))^2.
\end{equation} 
In the example above, the monomial is $x_{1,2}x_{6,4}x_{3,5}^2 x_{4,3}^2$, and we can introduce one of the main results:

\begin{theorem}\label{thm:thm2}
    The monomials $\{\sss(w) : w \in B_n \}$ form a spanning set of $\CC[\xx]/I_{B_n}$.
\end{theorem}

Before presenting the proof of Theorem ~\ref{thm:thm2}, we first introduce the following lemma:

\begin{lemma}
\label{lem:lem1}
    The following set descends to a spanning set of $\CC[\xx]/I_{B_n}$:
    \begin{equation}
        \{m(\RR_1)m(\RR_2)^2: \RR_1,\RR_2 \in [n] \times [n],  \RR=\RR_1 \cup \RR_2 \text{ is a rook placement}  \}.
    \end{equation}
\end{lemma}
\begin{proof}
    Since the ideal $I_{B_n}$ contains the cubes of all variables, as well as the products of any two variables in the same row or column, the result is immediate.
\end{proof}
With the lemma, we present the proof of Theorem~\ref{thm:thm2}.
\begin{proof}
    (of Theorem ~\ref{thm:thm2}) According to Lemma~\ref{lem:lem1}, we can fix a rook-placement $\RR_1$. By Remark~\ref{rk:rk1}, the set of $w \in B_n$ with $\nega(w)=\RR_1$ is nonempty, so it suffices to show 
    \begin{equation}
        \{m(\RR_1)m(\RR_2)^2: \RR=\RR_1 \cup \RR_2 \text{ is a rook placement} \}
    \end{equation} 
    lies, modulo the ideal $I_{B_n}$, in the linear span of
    \begin{equation}
        \{\sss(w) : w \in B_n,\nega(w)=\RR_1 \}.
    \end{equation}
    Let 
    \begin{equation}
        I = \{i: \text{there exists } j \text{ such that } (i,j) \in \RR_1 \},
    \end{equation} 
    and 
    \begin{equation}
        J = \{j: \text{there exists } i \text{ such that } (i,j) \in \RR_1 \}.
    \end{equation} 
    Fixing $i \in [n]\setminus I$,  since the sum of squares of all variables in the same row is in $I_{B_n}$, we have
    \begin{equation}
        \sum_{j \in [n] \setminus J}x_{i,j}^2 \equiv -\sum_{j \in J}x_{i,j}^2 \mod I_{B_n}.
    \end{equation}
    Similarly, fixing $j \in [n] \setminus J$, we also have the equation
    \begin{equation}
        \sum_{i \in [n] \setminus I}x_{i,j}^2 \equiv -\sum_{i \in I}x_{i,j}^2 \mod I_{B_n}.
    \end{equation}
    
    Thus, since $I_{B_n}$ contains the products of any two variables in the same row or column, we obtain the following relations:
    
    \begin{itemize}
        \item $m(\RR_1) \cdot \sum_{j \in [n] \setminus J} x_{i,j}^2 \equiv 0 \mod I_{B_n}$
        \item $m(\RR_1) \cdot \sum_{i \in [n] \setminus I} x_{i,j}^2 \equiv 0 \mod I_{B_n}$
    \end{itemize}
    We also have the following relations directly from the definition of $I_{B_n}$:
    
    \begin{itemize}
        \item $(x_{i,j}^2)^2 \equiv 0 \mod I_{B_n}$ for $i \in [n] \setminus I$, $j \in [n] \setminus J$;
        \item $x_{i,j}^2 \cdot x_{i,j'}^2 \equiv 0 \mod I_{B_n}$ for $i \in [n] \setminus I$, $j,j' \in [n] \setminus J$;
        \item $x_{i,j}^2 \cdot x_{i',j}^2 \equiv 0 \mod I_{B_n}$ for $i,i' \in [n] \setminus I$, $j \in [n] \setminus J$.
    \end{itemize}
    
    Then, if we look at the $(n - |\RR_1|) \times (n - |\RR_1|)$ matrix of variables, where each variable is $x_{i,j}^2$ for $i \in [n] \setminus I, j \in [n] \setminus J$, the collection of relations are exactly the ones we used to define $I_n$ in Definition~\ref{defn:defn1}. For a signed permutation $w$ with $\nega(w)=\RR_1$, if we delete the rows and columns that $\RR_1$ appears in, $\pos(w)$ will form a permutation in the remaining $(n - |I|) \times (n - |J|)$ grid. And since 
    \begin{equation}
        \{\pos(w):w \in B_n,\nega(w)=\RR_1\}
    \end{equation}
    ranges over all possible permutations in the grid labelled by $([n] \setminus I) \times ([n] \setminus J)$, according to Theorem~\ref{thm:thm1}, the proof is complete.
\end{proof}

\subsubsection{Standard Monomial Basis}
Viewing $B_n$ as $n \times n$ matrices with exactly one nonzero entry in each row or column, where the nonzero entries can be $\pm 1$, we have that the following polynomials all vanish on $B_n$: 
\begin{itemize}
    \item $x_{i,j}^3-x_{i,j}$ for $1 \leq i,j \leq n$;
    \item $x_{i,j} \cdot x_{i,j'}$ for $1 \leq i \leq n$, $1 \leq j < j' \leq n$;
    \item $x_{i,j} \cdot x_{i',j}$ for $1 \leq j \leq n$, $1 \leq i < i' \leq n$;
    \item $x_{i,1}^2+x_{i,2}^2+\dots+x_{i,n}^2-1$ for $1 \leq i \leq n$;
    \item $x_{1,j}^2+x_{2,j}^2+\dots+x_{n,j}^2-1$ for $1 \leq j \leq n$.
\end{itemize}
Let $J$ be the ideal generated by these polynomials, it follows immediately that $J \subseteq \II(B_n)$. Comparing these polynomials with the generators of $I_{B_n}$, it follows that $I_{B_n} \subseteq \gr J \subseteq \gr \II(B_n)$. We further improve this containment to an equality.

\begin{theorem}\label{thm:thm3}
    We have the equalities of ideals $I_{B_n}=\gr \II(B_n)$ and $J= \II(B_n)$. Moreover, the collection of monomials $\{\sss(w) : w \in B_n\}$ descends to a basis of the vector space $\CC[\xx]/I_{B_n}$. This is the standard monomial basis with respect to the Toeplitz order.
\end{theorem}

\begin{proof}
    Applying orbit harmonics, we have the following sequence of isomorphisms of (ungraded) vectors spaces
    \begin{equation}
    \CC[B_n] \cong \CC[\xx]/ \II(B_n) \cong \CC[\xx]/\gr \II(B_n). 
    \end{equation}
    With $I_{B_n} \subseteq \gr J \subseteq \gr \II(B_n)$, we have 
    \begin{equation}\label{eq:eqn3}
        2^n n! = \dim (\CC[\xx]/\gr \II(B_n)) \leq \dim (\CC[\xx]/\gr J) \leq \dim(\CC[\xx]/I_{B_n}).
    \end{equation}
    Since there is a unique pair $(\nega(w),\SSS(\pos(w))$ for each signed permutation $w$, the spanning set $\{\sss(w) : w \in B_n\}$ in Theorem~\ref{thm:thm2} has size $2^n n!$. This forces all inequalities in ~\eqref{eq:eqn3} to be equalities. This proves the equalities of ideals, and hence proves that $\{\sss(w) : w \in B_n\}$ descends to a basis of $\CC[\xx]/I_{B_n}$. 

    To prove that this is actually the standard monomial basis with respect to the Toeplitz order, let $\RR_1 $ be any rook placement, and let $f$ be a monomial with square-free part $m(\RR_1)$. We claim that if $f$ is in the standard monomial basis then $f$ is of the form $\sss(w)$ for some $w \in B_n$. To see this, recall that $f$ being in the standard monomial basis means that $f$ is not the initial monomial of any polynomial in $I_{B_n}$. Then, using the notation in the end of the proof of Theorem~\ref{thm:thm2}, the perfect square part $f/m(\RR_1)$ of $f$ must not be the initial monomial of any polynomial in $I_{n- |\RR_1|}$, where the variables are $x_{i,j}^2$ with $i \in [n] \setminus I$ and $j \in [n] \setminus J$. Then, according to the result on the standard monomial basis of $I_n$ by Rhoades ~\cite[Theorem 3.12]{rhoades2023increasing}, the claim is proved; and since the vector space $\CC[\xx]/I_{B_n}$ has dimension $2^n n!$, all monomials of the form $\sss(w)$ for $w \in B_n$ must be in the standard monomial basis.
\end{proof}
With the following two definitions, we will be able to find the Hilbert series of $\CC[\xx]/I_{B_n}$.

\begin{defn}\label{defn:defn4}
    Given $w \in B_n$, we define $\lis(\pos(w))$ to be the size of the largest subset $L$ of $\pos(w)$ such that given $(i,j)$ and $(i',j')$ in $L$, $i<i'$ implies $j<j'$. 
\end{defn}

\begin{example}
    Let $w \in B_7$ with one-line notation $w = [3^0 , 1^1 , 6^0, 4^0, 7^0 , 2^1 , 5^1]$, we have $\pos(w) = \{ (1,3), (3,6), (4,4), (5,7) \}$, and $\lis(\pos(w)) = 3$.
\end{example}
\vspace{0.1in}
\begin{defn}~\label{defn:b_n,k}
    For any pair of integers $n,k$, we define 
    \begin{equation}
        b_{n,k}=|\{ w \in B_n: 2 \lis(\pos(w)) + |\nega(w)| = k \}|
    \end{equation}
\end{defn}

From the definition, it is immediate that $b_{n,1}=0$ for $n \geq 2$, as $|\nega(w)| \geq 2$ or $\lis(\pos(w)) \geq 1$. Then, we have the following corollary.
\vspace{0.1in}
\begin{corollary}~\label{cor:hilbert}
    We have the Hilbert series of $\CC[\xx]/ I_{B_n}$
    \begin{equation}
        \Hilb( \CC[\xx]/ I_{B_n}; q) = b_{n,2n} + b_{n,2n-1} \cdot q + \cdots + b_{n,1} \cdot q^{2n-1}.
    \end{equation}
    
\end{corollary}
\begin{proof}
    For $w \in B_n$, 
    \begin{align*}
        \deg \sss(w) &= 2 |\SSS(\pos(w))|+|\nega(w)| \\
        &= 2 (|\pos(w)|-\lis(\pos(w)))+|\nega(w)| \\
        &= 2n - |\nega(w)| - 2 \lis(\pos(w)).
    \end{align*}
    Thus, $|\{w \in B_n : \deg \sss(w) = k\}| = b_{n,2n-k}$ and Theorem~\ref{thm:thm3} completes the proof.
\end{proof}

\begin{example}
    When $n=3$, the Hilbert series of $\CC[\xx]/I_{B_3}$ is
    \begin{equation}
        \Hilb( \CC[\xx]/I_{B_3}; q) = 1 + 9q + 22q^2 + 9q^3 + q^4. 
    \end{equation}
\end{example}

\subsection{Generalization to $\symm_{n,r}$}
Given an $r$-colored permutation $w = (\sigma,\kappa)$, we also draw it on the $n \times n$ grid. A point $(i,j)$ is labelled in color $c_l$ if $\sigma(i)=j$ and $\kappa(j)=l$, where $l$ ranges over $0,1,\dots,r-1$. We define the sets $C_l(w)$ to be the collection of points colored in $c_l$. Then, we do the shadow set construction on the set $C_0(w)$ and obtain the shadow set $\SSS(C_0(w))$, and we associate a monomial $\sss(w)$ to $w$, defined as
\begin{equation}
    \sss(w) = m(C_1(w)) \cdot m(C_2(w))^2 \cdot \dots \cdot m(C_{r-1}(w))^{r-1} \cdot m(\SSS(C_0(w)))^r.
\end{equation}
With this construction, we give the following theorem:
\begin{theorem}\label{thm:thm4}
    Let $I_{\symm_{n,r}}$ be defined as in Definition~\ref{defn:defn2}. The monomials $\{ \sss(w) : w \in \symm_{n,r} \}$ descends to a spanning set of $\CC[\xx]/I_{\symm_{n,r}}$.
\end{theorem}

The proof of Theorem~\ref{thm:thm4} is similar to Theorem~\ref{thm:thm2}. We need the following lemma, which is similar to Lemma~\ref{lem:lem1}.

\begin{lemma}\label{lem:spanning}
    The following set of monomials descends to a spanning set of $\CC[\xx]/I_{\symm_{n,r}}$:
    \begin{equation}~\label{eq:lem_r}
        \left\{ \prod_{l=1}^r m(\RR_l)^l : \RR_l \subset [n] \times [n], \RR = \bigcup_{l=1}^r \RR_l \text{ is a rook placement} \right\}
    \end{equation}
\end{lemma}

\begin{proof}
    Since the ideal $I_{\symm_{n,r}}$ contains the $r+1$-th power of all variables, as well as the products of any two variables in the same row or column, the result is immediate.
\end{proof}

With Lemma~\ref{lem:spanning}, we give the proof of Theorem~\ref{thm:thm4}.
\begin{proof}
    (of Theorem~\ref{thm:thm4}) According to Lemma~\ref{lem:spanning}, we can fix a collection of rook placements $\RR_1,\RR_2, \dots , \RR_{r-1}$ such that
    \begin{equation}
        \bigcup_{l=1}^{r-1} \RR_l
    \end{equation}
    is also a rook placement. Let $f$ be the monomial
    \begin{equation}
        f = \prod_{l=1}^{r-1} m(\RR_l)^l,
    \end{equation}
    it suffices to show 
    \begin{equation}
        \{f \cdot m(\RR_r)^r: \bigcup_{l=1}^r \RR_l  \text{ is a rook placement} \}
    \end{equation} 
    lies, modulo $I_{\symm_{n,r}}$, in the linear span of 
    \begin{equation}
        \{\sss(w) : w \in \symm_{n,r}, C_l(w) = \RR_l \text{ for } 1 \leq l \leq r-1 \}.
    \end{equation}
    Let 
    \begin{equation}
        I = \{i: \text{there exists } j \in [n], 1 \leq l \leq r-1 \text{ such that } (i,j) \in \RR_l \},
    \end{equation}
    and 
    \begin{equation}
        J = \{j: \text{there exists } i \in [n],1 \leq l \leq r-1 \text{ such that } (i,j) \in \RR_l \}.
    \end{equation} 
    Note that we have 
    \begin{equation}
        |I|=|J|=|\bigcup_{l=1}^{r-1}\RR_l|.
    \end{equation} 
    Fixing $i \in [n]\setminus I$, since the sum of $r$-th powers of all variables in the same row is in $I_{\symm_{n,r}}$, we have
    \begin{equation}
        \sum_{j \in [n] \setminus J}x_{i,j}^r \equiv -\sum_{j \in J}x_{i,j}^r \mod I_{\symm_{n,r}}.
    \end{equation}
    Similarly, fixing $j \in [n] \setminus J$, we also have the equation
    \begin{equation}
        \sum_{i \in [n] \setminus I}x_{i,j}^r \equiv -\sum_{i \in I}x_{i,j}^r \mod I_{\symm_{n,r}}.
    \end{equation}
    Thus, since $I_{\symm_{n,r}}$ contains the products of any two variables in the same row or column, we obtain the following relations:
    
    \begin{itemize}
        \item $f \cdot \sum_{j \in [n] \setminus J} x_{i,j}^r \equiv 0 \mod I_{\symm_{n,r}}$;
        \item $f \cdot \sum_{i \in [n] \setminus I} x_{i,j}^r \equiv 0 \mod I_{\symm_{n,r}}$.
    \end{itemize}
    We also have the following relations directly from the definition of $I_{\symm_{n,r}}$:
    
    \begin{itemize}
        \item $(x_{i,j}^r)^2 \equiv 0 \mod I_{\symm_{n,r}}$ for $i \in [n] \setminus I$, $j \in [n] \setminus J$;
        \item $x_{i,j}^r \cdot x_{i,j'}^r \equiv 0 \mod I_{\symm_{n,r}}$ for $i \in [n] \setminus I$, $j,j' \in [n] \setminus J$;
        \item $x_{i,j}^r \cdot x_{i',j}^r \equiv 0 \mod I_{\symm_{n,r}}$ for $i,i' \in [n] \setminus I$, $j \in [n] \setminus J$.
    \end{itemize}
    
    If we look at the $(n - |I|) \times (n - |J|)$ matrix of variables, where each variable is $(x_{i,j})^r$ for $i \in [n] \setminus I, j \in [n] \setminus J$, the collection of relations are exactly the ones we used to define $I_n$ in Definition~\ref{defn:defn1}. For a colored permutation $w$ with $C_l(w)=\RR_l$ for $1 \leq l \leq r-1$, if we delete the rows and columns with labels $I$ and $J$ respectively, $C_0(w)$ will form a permutation in the remaining $(n - |I|) \times (n - |J|)$ grid. Since 
    \begin{equation}
        \{C_0(w):w \in B_n; C_l(w) = \RR_l \}
    \end{equation}
    ranges over all possible permutations in the grid labelled by $([n] \setminus I) \times ([n] \setminus J)$, the proof is complete by Theorem~\ref{thm:thm1}.
\end{proof}
Viewing $\symm_{n,r}$ as the collection of $n \times n$ matrices with exactly one nonzero entry in each row and column, and every nonzero entry being an $r$-th root of unity, we have that the following polynomials all vanish on $\symm_{n,r}$: 
\begin{itemize}
    \item $x_{i,j}^{r+1}-x_{i,j}$ for $1 \leq i,j \leq n$;
    \item $x_{i,j} \cdot x_{i,j'}$ for $1 \leq i \leq n$, $1 \leq j < j' \leq n$;
    \item $x_{i,j} \cdot x_{i',j}$ for $1 \leq j \leq n$, $1 \leq i < i' \leq n$;
    \item $x_{i,1}^r+x_{i,2}^r+\dots+x_{i,n}^r -1$ for $1 \leq i \leq n$;
    \item $x_{1,j}^r+x_{2,j}^r+\dots+x_{n,j}^r -1$ for $1 \leq j \leq n$.
\end{itemize}
Let $J$ be the ideal generated by these polynomials, we have $J \subseteq \II(\symm_{n,r})$ and $I_{\symm_{n,r}} \subseteq \gr J \subseteq \gr \II(\symm_{n,r})$. Thus, we have a theorem that is an analogy to Theorem~\ref{thm:thm3}.

\begin{theorem}\label{thm:thm5}
    We have the equalities of ideals $I_{\symm_{n,r}}=\gr \II(\symm_{n,r})$ and $J = \II(\symm_{n,r})$. Moreover, the collection of monomials $\{ \sss(w): w\in \symm_{n,r} \}$ descends to a basis of the vector space $\CC[\xx]/I_{\symm_{n,r}}$. This is the standard monomial basis with respect to the Toeplitz order.
\end{theorem}

\begin{proof}
    Applying orbit harmonics, we have the following sequence of isomorphisms of (ungraded) vector spaces
    \begin{equation}
    \CC[\symm_{n,r}] \cong \CC[\xx]/ \II(\symm_{n,r}) \cong \CC[\xx]/\gr \II(\symm_{n,r}).   
    \end{equation}
    With $I_{\symm_{n,r}} \subseteq \gr J \subseteq \gr \II(\symm_{n,r})$, we have 
    \begin{equation}\label{eq:eqn5}
       r^n n! = \dim (\CC[\xx]/\gr \II(\symm_{n,r})) \leq \dim(\CC[\xx]/gr J) \leq \dim(\CC[\xx]/I_{\symm_{n,r}}).
    \end{equation}
    Since there is a unique tuple $(C_1(w), \dots C_{r-1} (w), \SSS(C_r(w))$ for each colored permutation $w$, the spanning set $\{\sss(w) : w \in \symm_{n,r}\}$ in Theorem~\ref{thm:thm4} has size $r^n n!$. This forces all inequalities in ~\eqref{eq:eqn5} to be equalities. This proves the equalities of ideals, and hence proves that $\{\sss(w) : w \in \symm_{n,r}\}$ descends to a basis of $\CC[\xx]/I_{\symm_{n,r}}$. 
    The proof that this is actually the standard monomial basis with respect to the Toeplitz order is very similar to the proof of Theorem~\ref{thm:thm3}.
\end{proof}

We also generalize the notion of longest increasing subsequence to the colored permutation group. We first give the following two definitions:
\begin{defn}
    Given $w \in \symm_{n,r}$, we define $\lis(C_0(w))$ to be the size of the largest subset $L$ of $C_0(w)$ such that given $(i,j)$ and $(i',j')$ in $L$, $i<i'$ implies $j<j'$.
\end{defn}
\begin{defn}\label{defn:c_n,r,k}
    For integers $n,r,k$, we define
    \begin{equation}
        c_{n,r,k} = \left| \left\{ w \in \symm_{n,r}: r \cdot \lis(C_0(w)) + \sum_{i=1}^{r-1} (r-i) \cdot |C_i(w)| = k\right\} \right|
    \end{equation}
\end{defn}

\begin{example}
    Let $n=6$, $r=4$, and $w = [4^2, 1^0, 5^2, 6^3, 2^1, 3^0 ]$. We have $C_0 (w) = \{ (2,1), (6,3) \}$, so $\lis(C_0(w)) = 2$, and we have 
    \begin{equation}
        4 \cdot \lis(C_0(w)) + 3|C_1(w)| + 2 |C_2(w)| + |C_3(w)| = 4 \times 2 +3 \times 1 + 2 \times 2 + 1 \times 1 = 16,
    \end{equation}
    so $w$ would contribute 16 to $c_{6,4,16}$.
\end{example}

With these two definitions, we give the Hilbert series of $\CC[\xx]/I_{\symm_{n,r}}$.
\begin{theorem}\label{thm:hilb_n,r}
    The Hilbert series of $\CC[\xx]/I_{\symm_{n,r}}$ is 
    \begin{equation}
        \Hilb(\CC[\xx]/I_{\symm_{n,r}};q) = \sum_{k=1}^{rn} c_{n,r,k} q^{rn-k}
    \end{equation}
\end{theorem}
\begin{proof}
    For $w \in \symm_{n,r}$,
    \begin{align*}
        \deg \sss(w) &= r \cdot |\SSS(C_0(w))| + \sum_{i=1}^{r-1} i \cdot |C_i(w)| \\
        &= r \cdot (C_0(w) - \lis(C_0(w))) +\sum_{i=1}^{r-1} i \cdot |C_i(w)| \\
        &= rn - r \cdot \lis(C_0(w)) - \sum_{i=1}^{r-1} (r-i) \cdot |C_i(w)|
    \end{align*}
Thus, $|\{w \in \symm_{n,r} : \deg \sss(w) = k\}| = c_{n,r,rn-k}$ and Theorem~\ref{thm:thm5} completes the proof.
\end{proof}

\section{Module Structure}\label{Module}

As mentioned in the Introduction, treating $\symm_{n,r}$ as the group of $n \times n$ matrices with exactly one nonzero entry in each row and column, and every nonzero entry being an $r$-th root of unity, we can let $\symm_{n,r} \times \symm_{n,r}$ act on $\CC[\xx]$, induced by its action on the matrix of variables
\begin{equation}
    (A,A') \cdot \xx = A \xx A'^{-1}.
\end{equation}
Since the collection of generators of $I_{\symm_{n,r}}$ is stable under this action, we have that $\CC[\xx]/I_{\symm_{n,r}}$ is an $\symm_{n,r} \times \symm_{n,r}$ module. By Theorem~\ref{thm:thm5}, we have an isomorphism and an equality of ungraded vector spaces
\begin{equation}\label{eq:ungraded}
    \CC[\symm_{n,r}] \cong \CC[\xx]/\gr \II(\symm_{n,r}) = \CC[\xx]/I_{\symm_{n,r}}.
\end{equation}
Since $\CC[\symm_{n,r}]$ is semi-simple, ~\eqref{eq:ungraded} upgrades to an isomorphism and an equality of ungraded $\symm_{n,r} \times \symm_{n,r}$ modules.

To investigate the graded structure of $\CC[\xx]/I_{\symm_{n,r}}$, we first introduce a lemma, which is a crucial part in orbit harmonics. Let $\x$ be a collection of variables and $\CC[\x]$ be the polynomial ring. Given a graded $\CC$-algebra $A = \bigoplus_{i \geq 0} A_i$, let $A_{\leq d} = \bigoplus_{i=0}^{d} A_i$ be the direct sum of its homogeneous components with degree less than or equal to $d$.

\begin{lemma}~\label{lem:lem2}
    Let $I$ be an ideal in $\CC[\x]$ and $\gr I$ be the associated graded ideal of it. Let $\mathcal{B} \subseteq \CC[\x]_{\leq d}$ be a collection of homogeneous polynomials with degree at most $d$. If $\mathcal{B}$ descends to a basis of the vector space $(\CC[\x]/gr I)_{\leq d}$, then $\mathcal{B}$ descends to a basis of the vector space $\CC[\x]_{\leq d}/(I \cap \CC[\x]_{\leq d})$.
\end{lemma}

\begin{proof}
    The proof can be found in the proof of a known result~\cite[Lemma 3.15]{rhoades2023increasing}.
\end{proof}

\subsection{Graded Module Structure of $\CC[\xx]/I_{B_n}$}

We first work on the case when $r=2$. In this case, the irreducible $B_n \times B_n$ modules are of form $V^{(\lambda,\mu)} \otimes V^{(\tl,\tm)}$, where $( \lambda,\mu )$ and $(\tl,\tm)$ are bipartitions of $n$.
Equation~\eqref{eq:ungraded} now reads:
\begin{equation}
    \CC[B_n] \cong \CC[\xx]/\gr \II(B_n) = \CC[\xx]/I_{B_n}.
\end{equation}
We give the main theorem of this section.

\begin{theorem}\label{thm:B_n}
    For any $k \geq 0$, the degree $k$ piece $(\CC[\xx]/I_{B_n})_{k}$ has graded $B_n \times B_n$ module structure
    \begin{equation}
        (\CC[\xx]/I_{B_n})_{k} \cong \bigoplus_{\substack{(\lambda,\mu) \vdash_2 n \\ 2 \lambda_1 + |\mu| \geq 2n-k}} V^{(\lambda,\mu)} \otimes V^{(\lambda,\mu)}.
    \end{equation}
\end{theorem}

\begin{proof}
    For $W$ any $B_n \times B_n$ module, $\End_{\CC}(W)$ is also a $B_n \times B_n$ module with the action
    \begin{equation}
        ((u,v) \cdot \varphi) (w) = u \cdot \varphi(v^{-1} \cdot w) \text{ for all } u,v \in B_n \text{, } \varphi \in \End_{\CC}(W), w \in W.
    \end{equation}
    
    We also have that $\End_{\CC}(W)$ is isomorphic to $W \otimes W^*$ as $B_n \times B_n$ modules, where $W^*$ is the dual module of $W$, via the isomorphism 
    \begin{equation}\label{eq:eqn4}
        W \otimes W^* \rightarrow \End_{\CC}(W), \quad a \otimes x \mapsto \varphi, \text{ with } \varphi(w) = x(w) \cdot a,
    \end{equation}
    where $a,w \in W$ and $x \in W^*$. Then, since all matrix representations of $B_n$ can be realized over the real numbers, we have that $W$ is isomorphic to $W^*$. So we have an isomorphism of $B_n \times B_n$ modules
    \begin{equation}\label{eq:End}
        \End_{\CC}(W) \cong W \otimes W.
    \end{equation}

    Since the algebra $\CC[B_N]$ is semi-simple, the {\em Artin-Wedderburn Theorem} gives the isomorphism of $\CC$-algebras
    \begin{equation}\label{eq:A-W}
        \Psi : \CC[B_n] \xrightarrow{ \, \, \sim \, \, } \bigoplus_{(\lambda,\mu) \vdash_2 n} \End_{\CC}(V^{(\lambda,\mu)}),
    \end{equation}
    where the $(\lambda,\mu)$-piece of $\Psi(a)$ maps $v$ to $a \cdot v$ for $a \in \CC[B_n]$, $v \in V^{(\lambda,\mu)}$. Since $\CC[B_n]$ is a $B_n \times B_n$ module via left and right action, ~\eqref{eq:End} and ~\eqref{eq:A-W} imply that we have the isomorphism of ungraded $B_n \times B_n$ modules
    \begin{equation}\label{eq:pf_chain}
        \CC[B_n] \cong \bigoplus_{(\lambda,\mu) \vdash_2 n} \End_{\CC}(V^{(\lambda,\mu)}) \cong \bigoplus_{(\lambda,\mu) \vdash_2 n} V^{(\lambda,\mu)} \otimes V^{(\lambda,\mu)}.
    \end{equation}

    Returning to the statement of the theorem, by an inductive argument, it suffices to prove that for any $k$, we have the isomorphism of ungraded $B_n \times B_n$ modules:
    \begin{equation}
        (\CC[\xx]/I_{B_n})_{\leq k} \cong \bigoplus_{\substack{(\lambda,\mu) \vdash_2 n \\ 2 \lambda_1 + |\mu| \geq 2n-k}} \End_{\CC}(V^{(\lambda,\mu)}).
    \end{equation}
    With Theorem~\ref{thm:thm3}, we have a chain of isomorphisms and an equality of ungraded $B_n \times B_n$ modules:
    \begin{equation}\label{eq:Bn-chain}
        \CC[B_n] \cong \CC[\xx]/\II(B_n) \cong \CC[\xx]/\gr \II(B_n) = \CC[\xx]/I_{B_n}.
    \end{equation}
    Then, consider the image of $\CC[\xx]_{\leq k}$ in $\CC[\xx]/\II(B_n)$, denoted by $L_k$, i.e.
    \begin{equation}\label{Bn:L_k}
        L_k = \mathrm{Image}(\CC[\xx]_{\leq k} \hookrightarrow \CC[\xx] \twoheadrightarrow \CC[\xx]/\II(B_n)  ).
    \end{equation}
    By Lemma~\ref{lem:lem1} and Lemma~\ref{lem:lem2}, we have that
    \begin{equation}
        L_k = \mathrm{span}_{\CC} \{ m(\RR_1)m(\RR_2)^2 + \II(B_n) : \RR_1 \cup \RR_2 \text{ is a rook placement and }  |\RR_1|+2|\RR_2| \leq k   \}.
    \end{equation}
    Then, Lemma~\ref{lem:lem2} implies that we have the isomorphism of $B_n \times B_n$ modules
    \begin{equation}
        L_k \cong (\CC[\xx]/I_{B_n})_{\leq k}.
    \end{equation}
    So we are reduced to showing that 
    \begin{equation}
        L_k \cong \bigoplus_{\substack{(\lambda,\mu) \vdash_2 n \\ 2 \lambda_1 + |\mu| \geq 2n-k}} \End_{\CC}(V^{(\lambda,\mu)}).
    \end{equation}
    
    For any $j \leq k$, consider the collection of monomials
    \begin{equation}
        M^j = \left\{ m_i^j = x_{1,1}^2 x_{2,2}^2 \cdots x_{i,i}^2 x_{i+1,i+1}x_{i+2,i+2} \cdots x_{j-i,j-i} : 0 \leq i \leq \lfloor j/2 \rfloor \right\}.
    \end{equation}
    For example, when $j=4$, the collection of monomials is
    \begin{equation}
        M^j = \left\{ x_{1,1}^2 x_{2,2}^2, x_{1,1}^2 x_{2,2} x_{3,3}, x_{1,1}x_{2,2}x_{3,3}x_{4,4} \right\}.
    \end{equation}
    Since nonzero monomials in $\CC[\xx]/I_{B_n}$ are cube-free, the degree $j$ monomials in $\CC[\xx]/I_{B_n}$ are generated by the action of $B_n \times B_n$ on $M^j$. Thus, $L_k$ is generated by the action of $B_n \times B_n$ on $\bigcup_{j \leq k}  M^j$. 
    
    Fix $0 \leq j \leq k$. For all $i$ satisfying $0 \leq i \leq \lfloor j/2 \rfloor$, embed $B_{n-j+i} \subseteq B_n$ by letting it act on the last $n-j+i$ letters. Let $\epsilon_i^j$ be the group algebra element
    \begin{equation}
        \epsilon_i^j =  (e+t_{1,-1})(e+t_{2,-2}) \cdots (e+t_{i,-i}) (e-t_{i+1,-i-1}) \cdots (e-t_{j-i,-j+i}) \left( \sum_{w \in B_{n-j+i}} w \right),
    \end{equation}    
    where $e$ denotes the identity element and $t_{l,-l}$ denotes the transposition in $B_n$ that interchanges $l$ and $-l$, and fixing all other elements. Identify $\CC[B_n]$ with the collection of functions from $B_n$ to $\CC$, via the rule
    \begin{equation}
        \sum_{w \in B_n} a_w w \mapsto f: f(w)=a_w.
    \end{equation}    
    Since $\CC[\xx]/\II(B_n)$ is naturally identified with functions from $B_n$ to $\CC$, we can identify $\epsilon_i^j$ with $m_i^j$. Let $\Sigma^j = \{ \epsilon_i^j : 0 \leq i \leq \lfloor j/2 \rfloor \}$, and let $J_k$ be the two-sided ideal generated by $\bigcup_{0 \leq j \leq k} \Sigma^j$. We have the identification
    \begin{equation}
        J_k = L_k.
    \end{equation}

    So we are now reduced to find the image of $J_k$ under the Artin-Wedderburn map in~\eqref{eq:A-W}. Note that the image of $J_k$ would be
    direct sum of endomorphisms of irreducible modules such that not all elements of $J_k$ act as the $0$ operator. Fixing $j \leq k$, we need to find the irreducible modules such that not all elements of $\Sigma_j$ act as $0$.

    Note that $\epsilon_i^j$ is naturally an element of the group algebra $\CC[B_{j-i} \times B_{n-j+i}]$. Given $V^{(\lambda,\mu)}$ an irreducible module of $B_n$ labelled by bipartition $(\lambda, \mu)$, it suffices to analyze the character of
    \begin{equation}
        \Res^{B_n}_{B_{j-i} \times B_{n-j+i}} V^{(\lambda,\mu)},
    \end{equation}
    evaluated on $\epsilon_i^j$. By the Branching Rule~\eqref{eq:branching},
    \begin{equation}
        \Res^{B_n}_{B_{j-i} \times B_{n-j+i}} \chi^{(\lambda,\mu)} = \sum_{(\tl, \tm)} \chi^{(\tl, \tm)} \times \chi^{(\lambda/\tl,\mu/\tm)},
    \end{equation}
    where the sum is over all $(\tl, \tm)$ that are bipartitions of $j-i$ and that $(\lambda/\tl,\mu/\tl)$ is well-defined. Denote
    \begin{equation}
        \eta_i^j =  (e+t_{1,-1})(e+t_{2,-2})\dots(e+t_{i,-i}) (e-t_{i+1,-i-1})\dots(e-t_{j-i,-j+i}),
    \end{equation}
    that is, $\epsilon_i^j = \eta_i^j \sum_{w \in B_{n-k+i}} w$, we have
    \begin{equation}
        \chi^{(\lambda,\mu)}(\epsilon^j_i) = \sum_{(\tl, \tm)} \chi^{(\tl, \tm)}(\eta_i^j) \chi^{(\lambda/\tl,\mu/\tm)} \left( \sum_{w \in B_{n-j+i}} w \right).
    \end{equation}

    Let $\theta$ be an irreducible character of $\ZZ/2\ZZ \cong B_1$, we have that $\theta (e + t_{l,-l}) \neq 0$ if and only if $\theta$ is the trivial character, and $\theta (e-t_{l,-l}) \neq 0$ if and only if $\theta $ is the signed character, i.e. $\theta(t_{l,-l}) = -1$. By the Generalized Murnaghan-Nakayama Rule~\eqref{eq:M-N}, we have that $\chi^{(\tl, \tm)}(\eta_i^j)$ is nonzero if and only if $|\tl|=i$ and $|\tm| = j-2i$. On the other hand, 
    \begin{equation}
        \chi^{(\lambda/\tl,\mu/\tm)} \left( \sum_{w \in B_{n-j+i}} w \right)
    \end{equation}
    is nonzero if and only if the trivial representation $\chi^{((n-j+i),\emptyset)}$ has non-zero multiplicity in $\chi^{(\lambda/\tl,\mu/\tm)}$. This forces $\mu/\tm = \emptyset$, and the diagram of $\lambda/\tl$ has no two boxes in the same row. Combining these conditions, we have that $\epsilon_i^j$ acts as a nonzero operator on $V^{(\lambda, \mu)}$ if and only if the following holds:
    \begin{itemize}
        \item $|\mu| = j-2i$, thus $|\lambda| = n-j+2i $;
        \item there exists $\tl \vdash i$ such that the diagram of $\lambda/\tl$ has no two boxes in the same column.
    \end{itemize}
    Note that the second condition is equivalent to
    \begin{itemize}
        \item the diagram of $\lambda$ has less than $i$ boxes beyond the first row, i.e. $\lambda_1 \geq |\lambda| - i = n-j+i$.
    \end{itemize}
    Since the first condition is an equality, the two conditions can be combined to
    \begin{equation}
        2 \lambda_1 + |\mu| \geq 2n-j.
    \end{equation}
    Since $j \leq k$ is arbitrary, we can conclude that elements of $J_k$ act as nonzero operators on all $V^{(\lambda,\mu)}$ satisfying
    \begin{equation}
        2\lambda_1 + |\mu| \geq 2n-k.
    \end{equation}
    This gives the isomorphism of $B_n \times B_n$ modules
    \begin{equation}
        J_k \cong \bigoplus_{\substack{(\lambda,\mu) \vdash_2 n \\ 2 \lambda_1 + |\mu| \geq 2n-k}} \End_{\CC}(V^{(\lambda,\mu)}),
    \end{equation}
    which completes the proof.
\end{proof}

\subsection{Graded Module Structure of $\CC[\xx]/I_{\symm_{n,r}}$}
We now investigate the graded module structure of $\CC[\xx]/I_{\symm_{n,r}}$ for general $r$. Given $\bl = (\lambda^0,\lambda^1,\dots,\lambda^{r-1})$ an $r$-partition of $n$, let $\bl'= (\lambda^0,\lambda^{r-1},\lambda^{r-2},\dots,\lambda^1)$ be the {\em dual $r$-partition of $\bl$}, obtained by fixing $\lambda^0$ and interchanging $\lambda^i$ with $\lambda^{r-i}$. As the dual module of a unitary matrix representation is its conjugate, by the generalized Murnaghan-Nakayama rule~\eqref{eq:M-N}, the dual module of $V^{\bl}$ is $V^{\bl'}$, and we have the following theorem.

\begin{theorem}\label{thm:colored}
    For any $k \geq 0$, the degree $k$ piece $(\CC[\xx]/I_{\symm_{n,r}})_{k}$ has graded $\symm_{n,r} \times \symm_{n,r}$ structure
    \begin{equation}
        (\CC[\xx]/I_{\symm_{n,r}})_{k} \cong \bigoplus_{\substack{  \bl  \vdash_r  n \\ r \lambda^0_1 + \sum_{i=1}^{r-1} i |\lambda^i| = rn-k  }} V^{\bl} \otimes V^{\bl'}
    \end{equation}
\end{theorem}
\begin{proof}
    Similar to the proof of Theorem~\ref{thm:B_n}, we have that for any $\symm_{n,r} \times \symm_{n,r}$ module $W$,
    \begin{equation}
        \End_{\CC}(W) \cong W \otimes W^*.
    \end{equation}
    Then, similar to~\eqref{eq:pf_chain}, we have the isomorphisms of $\symm_{n,r} \times \symm_{n,r}$ modules
    \begin{equation}
        \CC[\symm_{n,r}] \cong \bigoplus_{\bl \vdash_r n} \End_{\CC}(V^{\bl}) \cong \bigoplus_{\bl \vdash_r n} V^{\bl} \otimes ( V^{\bl} ) ^* \cong \bigoplus_{\bl \vdash_r n}V^{\bl}  \otimes V^{\bl'}.
    \end{equation}    
    By Theorem~\ref{thm:thm5}, we have a chain of isomorphisms and an equality of ungraded $\symm_{n,r} \times \symm_{n,r}$ modules:
    \begin{equation}
        \CC[\symm_{n,r}] \cong \CC[\xx]/\II(\symm_{n,r}) \cong \CC[\xx]/ \gr \II(\symm_{n,r}) = \CC[\xx]/I_{\symm_{n,r}}.
    \end{equation}
    Again, let $L_k$ be the image of $\CC[\xx]_{\leq k}$ to $\CC[\xx]/\II ( \symm_{n,r})$, that is,
    \begin{equation}
        L_k = \mathrm{Image}(\CC[\xx]_{\leq k} \hookrightarrow \CC[\xx] \twoheadrightarrow \CC[\xx]/\II(\symm_{n,r})  ).
    \end{equation}
    By Lemma~\ref{lem:lem2}, we have that 
    \begin{equation}~\label{eq:L_k}
         L_k = \mathrm{span}_{\CC} \left\{ \left( \prod_{i=1}^{r} m(\RR_i)^i \right) + \II(\symm_{n,r}) : \bigcup_{i=1}^r \RR_i \text{ is a rook placement and }  \sum_{i=1}^{r} i \cdot |\RR_i| \leq k   \right\}.
    \end{equation}    
    Applying Lemma~\ref{lem:lem2} again, we have the isomorphism of ungraded $\symm_{n,r} \times \symm_{n,r}$ modules
    \begin{equation}
        L_k \cong (\CC[\xx]/I_{\symm_{n,r}})_{\leq k},
    \end{equation}    
    qnd we are reduced to showing that
    \begin{equation}
        L_k \cong \bigoplus_{\substack{\bl \vdash_r n \\ r \lambda^0_1 + \sum_{i=1}^{r-1} i |\lambda^i| \geq rn-k }} \End_{\CC}(V^{\bl}).
    \end{equation}
    
    Let $M_k$ be the subset of the spanning set in ~\eqref{eq:L_k}, containing all monomials
    \begin{equation}
        \prod_{i=1}^{r} m(\RR_i)^i 
    \end{equation}
    where the rook placements $\{\RR_i \}_{i=1}^r$ satisfy the criteria in ~\eqref{eq:L_k}, and for all $j \leq r$,
    \begin{equation}~\label{eq:rk-criteria}
        \bigcup_{i=1}^j \RR_i = \{(a,a) : 1 \leq a \leq \sum_{i=1}^j |\RR_i| \}.  
    \end{equation}
    In other words, in the $n \times n $ grid, the rook placements $\{\RR_i\}$ are arranged on the diagonals, from bottom left to top right. For example, when $k=5$ and $r=3$, we have
    \begin{equation}
    M_k = \{x_{1,1}x_{2,2}x_{3,3}x_{4,4}x_{5,5}, x_{1,1}x_{2,2}x_{3,3}x_{4,4}^2, x_{1,1}x_{2,2}x_{3,3}^3, x_{1,1}x_{2,2}^2x_{3,3}^2 , x_{1,1}^2 x_{2,2}^3 \}.
    \end{equation}
    Monomials in $L_k$ are generated by monomials in $M_k$ by two sided action of $\symm_{n,r} \times \symm_{n,r}$. 

    Again, identifying $\CC[\symm_{n,r}]$ and $\CC[\xx]/\II(\symm_{n,r})$ with functions from $\symm_{n,r}$ to $\CC$, we can identify $\x_{i,i}^j$ in $\CC[\xx]/\II(\symm_{n,r})$ with the group algebra element
    \begin{equation}
        (e+ \omega^j (i^0,i^1) + \omega ^{2j} (i^0,i^2)+ \cdots +\omega^{(r-1)j} (i^0, i^{r-1})) \sum_{w \in \symm_{n-1,r}} w   ,  
    \end{equation}
    where $\symm_{n-1,r}$ is embedded in $\symm_{n,r}$ by acting on the last $n-1$ letters, $e$ denotes the identity in $\symm_{n,r}$, $\omega$ is the primitive $r$-th root of unity in $\CC$, and $ (i^0, i^l)$ denotes the element in $\symm_{n,r}$ that maps $i^0$ to $i^l$ while fixing all other elements. Let $\Sigma_k$ be the collection of group algebra elements identified with monomials in $M_k$, and let $J_k$ be the two-sided ideal generated by $J_k$. We have identification
    \begin{equation}
        L_k = J_k.
    \end{equation}

    So we are reduced to finding $r$-partitions $\bl$ such that not all elements of $J_k$ acts as zero operator. Denote
    \begin{equation}
        \epsilon_i^j = e+ \omega^j (i^0,i^1) + \omega ^{2j} (i^0,i^2)+ \cdots +\omega^{(r-1)j} (i^0, i^{r-1}).
    \end{equation}
    For rook placements $\{\RR_j \}_{j=1}^r$ satisfying the criteria in ~\eqref{eq:L_k} and~\eqref{eq:rk-criteria}, let $\RR=\bigcup_{j=1}^r \RR_j$, we have that the group algebra element that is identified with $\prod_{j=1}^{r} m(\RR_j)^i$ is
    \begin{equation}
        \epsilon = \left( \prod_{j=1}^r \prod_{(i,i) \in \RR_j} \epsilon_i^j \right) \sum_{w \in \symm_{n-|\RR|,r}}w.
    \end{equation}

    Let $\eta = \prod_{j=1}^r \prod_{(i,i) \in \RR_j} \epsilon_i^j $, we can embed $\symm_{n-|\RR|,r}$ in $\symm_{n,r}$ by letting it act on the last $n-|\RR|$ letters. Then, by the Branching Rule~\eqref{eq:branching}, we have
    \begin{equation}
        \chi^{\bl}(\epsilon) = \sum_{\bm \vdash_r |\RR|} \chi^{\bm}(\eta) \chi^{\bl/\bm} \left( \sum_{w \in \symm_{n-|\RR|,r}}w \right).
    \end{equation}

    Let $\theta$ be any irreducible $\CC$-character of $\symm_{1,r} \cong \ZZ/r \ZZ$. Note that we have $\theta (\epsilon^j_i) \neq 0$ if and only if $\theta((i^0,i^1)) = \omega^{-j}$. Thus, by the generalized Murnaghan-Nakayama rule~\eqref{eq:M-N}, we have that for $\bm = (\mu^0,\mu^1,\dots\mu^{r-1}) \vdash_r |\RR|$, $\chi^{\bm}(\eta)$ is nonzero if and only if $|\mu^j| = |\RR_{r-j}|$. On the other hand, 
    \begin{equation}
        \chi^{\bl/\bm} \left( \sum_{w \in \symm_{n-|\RR|,r}}w \right)
    \end{equation}
    is nonzero if and only if the trivial character has nonzero multiplicity in $\chi^{\bl/\bm}$, which indicates that $|\mu^j| = |\lambda^j|$ for $1 \leq j \leq r-1$, and there are no two boxes in the same column in the diagram of $\lambda^0/\mu^0$. That is, there are less than $|\mu^0|=|\RR_r|$ boxes beyond the first row in the diagram of $\lambda^0$.

    Let $d=\sum_{j=1}^r j \cdot |\RR_j|$ be the degree of the monomial $\prod_{j=1}^{r} m(\RR_j)^i$. Combining these conditions, we have that $\epsilon$ acts as a nonzero operator on $V^{\bl}$ if there exists $\bm = (\mu^0,\mu^1,\dots,\mu^{r-1})$ such that the following holds:
    \begin{enumerate}
        \item $r|\mu^0|+(r-1) |\mu^1| + \dots + |\mu^{r-1}| = d$;
        \item $\lambda^j = \mu^j$ for $j \neq 0$;
        \item diagram of $\lambda^0/\mu^0$ has no two boxes in the same column.
    \end{enumerate}
    Note that the third condition is equivalent to 
    \begin{equation}
        \lambda^0_1 \geq |\lambda^0| - |\mu^0| = n-|\lambda^1|-|\lambda^2|- \dots -|\lambda^{r-1}|-|\mu^0|.
    \end{equation}
    Multiplying both sides by $r$, we get
    \begin{equation}\label{eq:final}
        r \lambda^0_1 \geq rn-r |\lambda^1| - r|\lambda^2|- \dots -r|\lambda^{r-1}| -r|\mu^0|.
    \end{equation}     
    Then, the first and second condition combined gives that
    \begin{equation}
        r|\mu^0| = d- (r-1)|\lambda^1|-(r-2)|\lambda^2|- \dots -|\lambda^{r-1}|.
    \end{equation}
    Substituting this into the inequality~\eqref{eq:final}, we have 
    \begin{equation}
        r \lambda^0_1 + |\lambda^1| + 2|\lambda^2| + \dots + (r-1)|\lambda^{r-1}| \geq rn-d. 
    \end{equation}
    This shows the isomorphism of $\symm_{n,r} \times \symm_{n,r}$ modules
    \begin{equation}
        J_k \cong \bigoplus_{\substack{\bl \vdash_r n \\ r \lambda^0_1 + \sum_{i=1}^{r-1} i |\lambda^i| \geq rn-k }} \End_{\CC}(V^{\bl}),
    \end{equation}
    which completes the proof.

\end{proof}

\section{Conclusions and Conjectures}\label{Conjecture}
Given a (possibly infinite) sequence ${a_i}$ of positive real numbers, we say that the sequence is {\em log-concave} if for all positive integer $i$, we have
\begin{equation}
    a_i \cdot a_{i+2} \leq a_{i+1}^2;
\end{equation}
and we say the sequence is {\em uni-modal} if there exists a positive integer $k$ such that
\begin{equation}
    a_i \leq a_{i+1} \text{ for } i<k; \quad a_i \geq a_{i+1} \text{ for } i \geq k.
\end{equation}
    In the case of Hilbert series, we say a Hilbert series is log-concave or uni-modal if its sequence of coefficients is log-concave or uni-modal. Recall that the Hilbert series of $\CC[\xx]/I_n$ is
\begin{equation}
    \Hilb( \CC[\xx] / I_n; q) = a_{n,n} + a_{n,n-1} \cdot q + a_{n,n-2} \cdot q^2  + \dots + a_{n,1} \cdot q^{n-1},
\end{equation}
where $a_{n,k}$ counts the number of permutations $w \in \symm_n$ with $\lis(w) = k$. Chen conjectured~\cite{chen} that the sequence $\{a_{n,k}\}$ is log-concave, that is, fixing an integer $n$, for all $k$ such that $2 \leq k \leq n-1$, 
\begin{equation}
    a_{n,k}^2 \geq a_{n,k-1} \cdot a_{n,k+1}.
\end{equation}

We found some interesting patterns when we tried to generalize this to the Hilbert series of $\CC[\xx]/I_{B_n}$. Recall that we showed in Corollary~\ref{cor:hilbert} that the Hilbert series of $\CC[\xx]/I_{B_n}$ is
\begin{equation}
    \Hilb( \CC[\xx]/ I_{B_n}; q) = b_{n,2n} + b_{n,2n-1} \cdot q + \cdots + b_{n,1} \cdot q^{2n-1},
\end{equation}
where $\{b_{n,k}\}_{k=1}^{2n-2}$ is defined in Definition~\ref{defn:b_n,k}. By testing 
$n$ up to $40$, we found that 
\begin{itemize}
    \item for $n \leq 8$, the sequence $\{b_{n,k}\}$ is log-concave;
    \item for $9 \leq n \leq 17$, we have $b_{n,k}^2 < b_{n,k-1} \cdot b_{n+1}$ when $k=3$, and log-concavity holds for the remaining of the sequence;
    \item for $18 \leq n \leq 27$, we have $b_{n,k}^2 < b_{n,k-1} \cdot b_{n+1}$ when $k=3 \text{ and } 5$, and log-concavity holds for the remaining of the sequence;
     \item for $27 \leq n \leq 40$, we have $b_{n,k}^2 < b_{n,k-1} \cdot b_{n+1}$ when $k=3,5, \text{ and } 7$, and log-concavity holds for the remaining of the sequence.
\end{itemize}
From this observation, we conjecture that the Hilbert series of $\CC[\xx]/I_{B_n}$ is ``almost log-concave": log-concavity breaks for some small $k$'s, but holds for the majority of the sequence.

In~\cite{baik}, Baik, Deift, and Johansson proved that as $n \rightarrow \infty$, the distribution function for $\lis$ on $\symm_n$, with proper rescaling and recentering, converges to the Tracy-Widom distribution of the largest eigenvalue of a random GUE matrix. That is, as $n \rightarrow \infty$, the histogram of $\{a_{n,k}\}$ for a fixed $n$ will converge in distribution (with rescaling and recentering) to the Tracy-Widom distribution. In Figure~\ref{fig:1}, we give the histogram of $\{a_{n,k}\}$ when $n=65$. The horizontal axis corresponds to values of $k$, and the height of each bar represents the corresponding $a_{n,k}$.
\begin{figure}
    \centering
    \includegraphics[width=10cm]{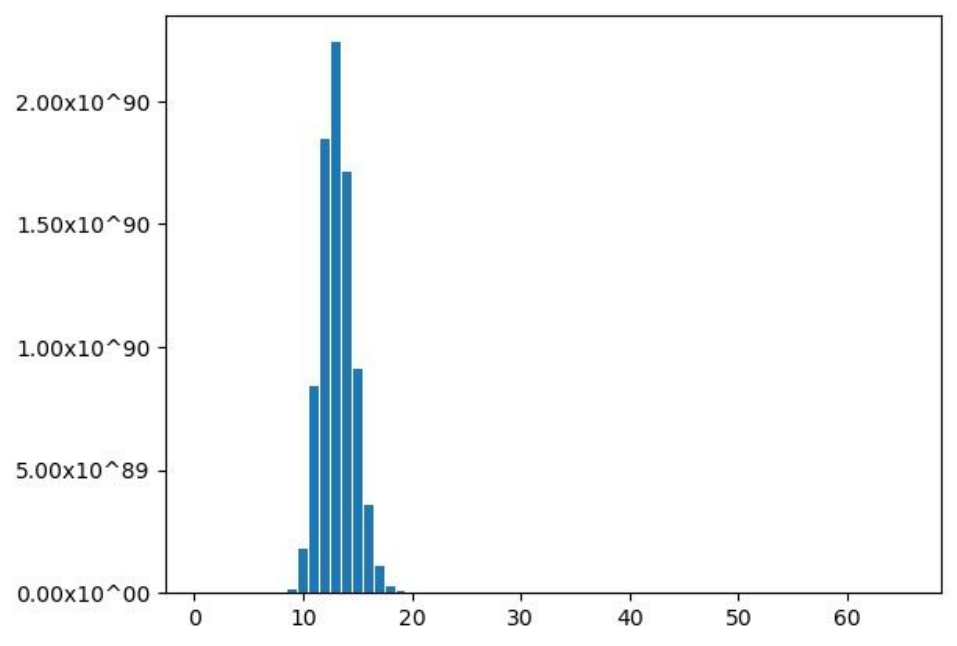}
    \caption{$a_{n,k}$ when $n=65$}
    \label{fig:1}
\end{figure}

One possible direction for future studies is to study the limit of the probability distributions corresponding to $\{b_{n,k}\}$ as $n \rightarrow \infty$. In analogy with the case of $\{a_{n,k}\}$, one could hope that we have convergence to a natural probability distribution. In Figure~\ref{fig:2}, we give the histogram of $\{b_{n,k}\}$ when $n=40$. The horizontal axis corresponds to values of $k$, and the height of each bar represents the corresponding $b_{n,k}$.

\begin{figure}
    \centering
    \includegraphics[width=10cm]{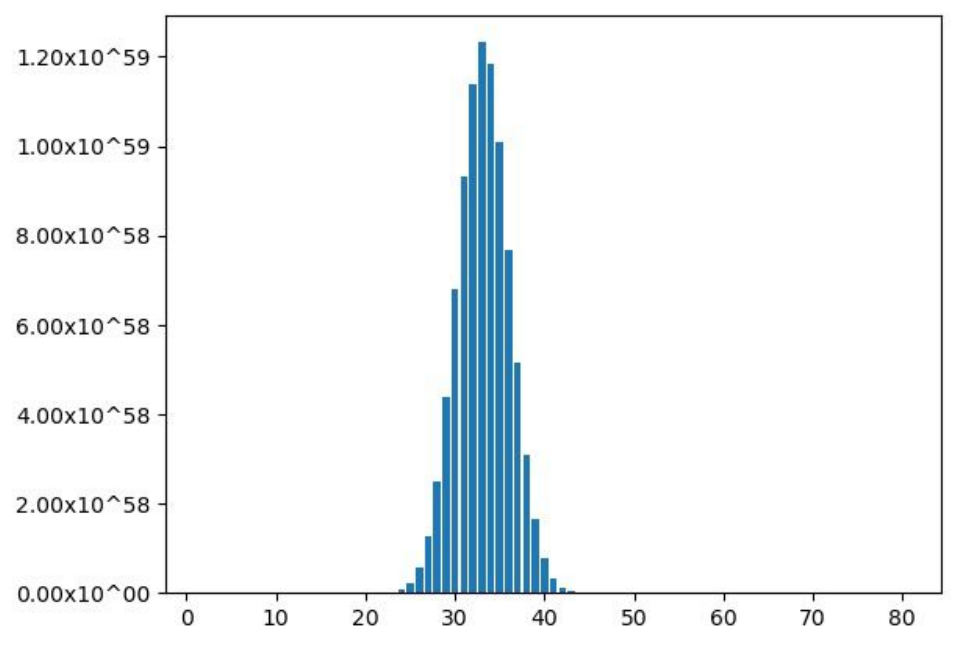}
    \caption{$b_{n,k}$ when $n=40$}
    \label{fig:2}
\end{figure}
Our testing on uni-modality also yields a very nice pattern, and we give the following conjecture:
\begin{conjecture}\label{conj}
    For any integers $n,r$, the Hilbert series of $\CC[\xx]/I_{\symm_{n,r}}$ is uni-modal, that is, there exists $k$ such that
    \begin{equation}
        c_{n,r,d} \leq c_{n,r,d+1} \text{ for } d < k \text{, and } c_{n,r,d} \geq c_{n,r,d+1} \text{ for } d \geq k.
    \end{equation}   
\end{conjecture}
Conjecture~\ref{conj} is tested to be true for $r \leq 4$ and $n \leq 20$.

As mentioned in~\cite{rhoades2023increasing}, orbit harmonics played a key tool in our analysis. We treated the colored permutation group $\symm_{n,r}$ as a collection of $n \times n$ matrices, and proved that $I_{\symm_{n,r}} = \gr \II(\symm_{n,r})$. We provide one possible direction for future studies below.

A {\em complex reflection} on a finite-dimensional vector space $V$ is an element of finite order in $\gl(V)$ which fixes a hyperplane pointwise. A {\em complex reflection group} $W$ is a finite group generated by complex reflections. Shepherd and Todd showed in~\cite{shephard_todd_1954} that any complex reflection group is either one of the $34$ exceptional cases, or in the form $G(r,p,n)$, where $r,p,n$ are integers with $p$ divides r, and the group $G(r,p,n)$ consists of $n \times n$ matrices with 
\begin{itemize}
    \item exactly one nonzero entry in each row or column;
    \item each nonzero entry is an $r$-th root of unity;
    \item the product of all the nonzero-entries is an $r/p$-th root of unity.
\end{itemize}
When $p=1$, $G(r,1,n)$ is the colored permutation group $\symm_{n,r}$, and we calculated $\gr \II(G(r,1,n)) = I_{\symm_{n,r}}$ in this paper. It may be interesting to study $\gr \II(G(r,p,n))$ for different $p$'s that divides $r$. For example, $G(2,2,n)$ coincides with the Coxeter group of type $D$, and it would be a good object for future study.

Another possible direction is on the equivariant log-concavity of $\CC[\xx]/I_{\symm_{n,r}}$. In~\cite{rhoades2023increasing}, Rhoades conjectured that with $\symm_n \times \symm_n$ acting on $(\CC[\xx])_d$ by independent row and column operation, there exists an injection
\begin{equation}
    (\CC[\xx]/I_n)_{d+1} \otimes (\CC[\xx]/I_n)_{d-1} \xhookrightarrow{\phi} (\CC[\xx]/I_n)_d \otimes (\CC[\xx]/I_n)_d
\end{equation}
such that given $(w,v) \in \symm_n \times \symm_n$, $f \in (\CC[\xx]/I_n)_{d+1}$, $g \in (\CC[\xx]/I_n)_{d-1}$,
\begin{equation}
    \phi \left( (w,v) \cdot (f \otimes g) \right) = (w,v) \cdot \left(\phi (f \otimes g)\right).
\end{equation}
This conjecture would imply both Chen's Conjecture~\cite{chen} and the Novak-Rhoades Conjecture~\cite{novak2020increasing}. As the Hilbert series of $\CC[\xx]/I_{\symm_{n,r}}$ is not always log-concave, it would be interesting to study given $\symm_{n,r}$, for which $d$'s we can have an injection
\begin{equation}
    (\CC[\xx]/I_{\symm_{n,r}})_{d+1} \otimes (\CC[\xx]/I_{\symm_{n,r}})_{d-1} \xhookrightarrow{} (\CC[\xx]/I_{\symm_{n,r}})_d \otimes (\CC[\xx]/I_{\symm_{n,r}})_d
\end{equation}
that commutes with the diagonal action of $\symm_{n,r} \times \symm_{n,r}$. Our calculation shows that when $n=1$ such an injection exists for any $r,d$.

\section{Acknowledgments}
The author is very grateful to Brendon Rhoades for proof-reading this paper and providing a lot of useful suggestions.
\printbibliography
\end{document}